\documentclass{amsart}
\usepackage{hyperref,math,mathabx,mathtools,amsrefs,mathrsfs,bbm}
\usepackage[a4paper]{geometry}
\newcommand\K{{\mathbb K}}
\newcommand\A{{\mathscr A}}
\newcommand\AB{{\mathscr B}}
\newcommand\AJ{{\mathscr J}}
\newcommand\id{{\mathbbm 1}}
\numberwithin{equation}{section}

\newcommand\LieI{{\mathfrak I}}
\newcommand\LieJ{{\mathfrak J}}
\newcommand\LieK{{\mathfrak K}}
\newcommand\LieL{{\mathfrak L}}
\newcommand\LieM{{\mathfrak M}}
\newcommand\LieQ{{\mathfrak Q}}
\newcommand\LieW{{\mathfrak W}}

\newcommand\gr{\operatorname{gr}}
\newcommand\Der{\operatorname{Der}}

\newtheorem{mainthm}{Theorem}
\newtheorem{maincor}[mainthm]{Corollary}

\theoremstyle{definition}\newtheorem*{maindefn}{Definition}

\begin{document}
\title{Amenability of Lie, Group and Hopf algebras}
\author{Laurent Bartholdi}
\date{September 3, 2026}
\address{Institut Camille Jordan, Université de Lyon 1 \textit{and} Section de Mathématiques, Université de Genève}
\email{laurent.bartholdi@gmail.com}

\thanks{The author gratefully acknowledges support from the SNSF Advanced Grant TMAG-2\_216487/1.}
\begin{abstract}
  We propose to define amenability of a Lie algebra by the existence of almost-invariant finite-dimensional subcoalgebras of its universal enveloping algebra.

  More generally, a module coalgebra of a cocommutative Hopf algebra is amenable if it admits almost-invariant finite-dimensional subcoalgebras. We prove a coalgebraic rounding theorem: almost-invariant finite-dimensional subspaces can be replaced, without loss in the Følner constant, by almost-invariant finite-dimensional subcoalgebras. For Lie algebras this implies that our definition is equivalent to Elek's amenability of the left regular module of the universal enveloping algebra seen merely as an associative algebra. For groups this recovers the result that a group is amenable if and only if its group ring is algebraically amenable. It furthermore shows that, for every amenable group, all its nonzero modules are amenable, thus proving an assertion by Gromov.

  We prove that amenable Hopf algebras are closed under taking subalgebras, quotients, cleft extensions, and directed unions, and that every Hopf algebra locally of subexponential growth is amenable. We give examples of amenable Lie algebras which are not elementarily amenable. Finally, we show that amenability passes to the associated graded Hopf-module coalgebra.
\end{abstract}
\maketitle

\begin{maindefn}
  Let $\LieL$ be a Lie algebra, presumed infinite-dimensional. It is \emph{amenable} if for every finite-dimensional $F\le \LieL$ and every $\epsilon>0$ there exists a finite-dimensional subcoalgebra $E$ of the universal enveloping algebra $U(\LieL)$ with
  \pushQED{\qed}
\[\dim\frac{E+F E}{E}<\epsilon\dim E.\qedhere\]
\end{maindefn}
In other words, $U(\LieL)$ contains almost-invariant finite-dimensional subcoalgebras. One purpose of this article is to convince you, reader, that this is the ``right'' definition of amenability for Lie algebras, and that it has the expected permanence properties as well as natural analogies to amenability of groups.

It is useful to state our first results in the more general setting of Hopf-module coalgebras. Let $H$ be a cocommutative Hopf algebra, and let $M$ be a coalgebra which is a left $H$-module, such that the action $m\colon H\otimes M\to M$ is a coalgebra map; equivalently,
\[
\Delta_M\circ m=(m\otimes m)\circ\Delta_{H\otimes M},
\qquad
\varepsilon_M\circ m=\varepsilon_H\otimes\varepsilon_M.
\]
We call $M$ \emph{amenable} if for every finite-dimensional $F\le H$ and every $\epsilon>0$ there exists a finite-dimensional subcoalgebra $E\le M$ with
\[\dim\frac{E+F E}{E}<\epsilon\dim E.\]

Elek introduced in~\cite{elek:amenaa} a notion of amenability for associative algebras: for every finite-dimensional $F\le H$ and every $\epsilon>0$ there must exist a finite-dimensional subspace $E\le M$ with
\[\dim\frac{E+F E}{E}<\epsilon\dim E.\]
We refer to this notion as \emph{algebraic amenability} to distinguish it from the one we consider for Hopf algebras.

\begin{mainthm}[= Theorem~\ref{thm:module-coalgebra-rounding}]\label{thm:hopf}
  A Hopf-module coalgebra is amenable if and only if it is algebraically amenable, as an ordinary associative-algebra module.
\end{mainthm}

\begin{mainthm}[= Theorem~\ref{thm:module-coalgebra-quotient}]\label{thm:quo}
  Every Hopf-module coalgebra quotient of an amenable module is amenable.
\end{mainthm}

The Hopf algebra $H$ itself is called \emph{amenable} if its regular module ${}_H H$ is amenable; equivalently,
\begin{maincor}
  The Hopf algebra $H$ is amenable if and only if every nonzero $H$-module coalgebra is amenable.
\end{maincor}
\begin{proof}
  ``if'' is obvious, taking the module ${}_H H$. Conversely, for a nonzero $H$-module coalgebra $M$, give $N\coloneqq H\otimes M$ the $H$-action on the first factor. As an $H$-module it is a direct sum of $\dim(M)>0$ copies of $H$, so is algebraically amenable, and therefore amenable by Theorem~\ref{thm:hopf}. The action map $N\to M$, $h\otimes m\mapsto h\cdot m$, is an $H$-module surjective coalgebra map, so $M$ is amenable by Theorem~\ref{thm:quo}.
\end{proof}

A Hopf algebra $H$ is \emph{locally of subexponential growth} if for every finite-dimensional $F\le H$ the dimensions $\dim F^n$ grow subexponentially in $n$.
\begin{mainthm}[= Theorems~\ref{thm:hopf-subalgebra}, \ref{thm:union}, \ref{thm:subexp}]
  Every Hopf subalgebra of an amenable cocommutative Hopf algebra is amenable.

  Every directed union of amenable Hopf algebras is amenable.

  Every Hopf algebra locally of subexponential growth (for example, commutative) is amenable.  
\end{mainthm}

\begin{mainthm}[= Theorem~\ref{thm:hopf-extension}]
  Let
  \[
    \Bbbk\longrightarrow A\longrightarrow B
    \overset{\pi}{\longrightarrow}C\longrightarrow\Bbbk
  \]
  be a \emph{cleft exact sequence} of cocommutative Hopf algebras, namely an exact sequence in which $\pi$ has a coalgebra section and $A=\{b\in B: (\id\otimes\pi)\Delta(b)=b\otimes 1\}$.

  Then $B$ is amenable if and only if $A$ and $C$ are amenable.
\end{mainthm}

The two main examples of cocommutative Hopf-module coalgebras are $H=\Bbbk G$ a group ring, with $M=\Bbbk X$ the linear span of a $G$-set, and $H=U(\LieL)$ a universal enveloping algebra, with $M=\operatorname{Sym}(\mathfrak X)$ for an $\LieL$-module $\mathfrak X$. In the first case, we recover the main result of~\cite{bartholdi:aa1} (see also~\cite{gromov:linear}*{\S3.6}): $\Bbbk X$ has almost-invariant finite-dimensional subspaces if and only if $X$ has almost-invariant finite subsets. For the regular module $M=U(\LieL)$, Theorem~\ref{thm:hopf}, together with the generator test proved below, shows that amenability of $\LieL$ may equivalently be defined using arbitrary finite-dimensional Følner subspaces of $U(\LieL)$.

In~\cite{gromov:linear}*{\S1.8, page~509}, Gromov writes that
``obviously every linear action of an amenable group is amenable''. This does not seem obvious: in a $G$-module $\Bbbk G/I$ it could theoretically happen that every Følner set in $G$ spans a very low-dimensional subspace modulo $I$, yet its boundary remains high-dimensional. Nevertheless, we prove Gromov's assertion, in the following strengthened form:
\begin{mainthm}[= Theorem~\ref{thm:quotients-of-permutation-modules}]
  Let $X$ be a left $G$-set and let $q\colon\Bbbk X\twoheadrightarrow V$ be a quotient of left $G$-modules. Suppose that $q(x)\neq0$ for some $x\in X$ whose orbit $G x$ is amenable. Then $V$ is amenable.

In particular, if $G$ is amenable, then every nonzero left $\Bbbk G$-module is amenable.
\end{mainthm}

Returning to Lie algebras, we prove:
\begin{mainthm}[= Theorems~\ref{thm:liesubexp}, \ref{thm:subalgebra}, \ref{thm:extension}, \ref{thm:lieunion}]
  Every Lie algebra locally of subexponential growth (e.g., Abelian) is amenable.
  
  Every subalgebra and quotient algebra of an amenable Lie algebra is amenable.

  Every extension and directed union of amenable Lie algebras is amenable.
\end{mainthm}

Following von Neumann~\cite{vneumann:masses} and Chou~\cite{chou:elementary}, we define \emph{elementarily amenable Lie algebras} as the smallest class \textsf{EL} of Lie algebras containing finite-dimensional and Abelian Lie algebras and closed under taking subalgebras, quotients, extensions, and directed unions; and \emph{subexponentially amenable Lie algebras} as the smallest class \textsf{SL} of Lie algebras containing Lie algebras of subexponential growth and closed under the same operations. Thus
\[
\textsf{EL}\subseteq\textsf{SL}\subseteq\textsf{AL},
\]
where \textsf{AL} denotes the class of amenable Lie algebras. Since there are simple Lie algebras of subexponential growth, we obtain:
\begin{mainthm}[= Theorems~\ref{thm:amennotelem1}, \ref{thm:amennotelem2}]
  The classes \textsf{EL} and \textsf{SL} are distinct.
\end{mainthm}
We do not know, at present, whether \textsf{SL} and \textsf{AL} are distinct, as is the case for groups~\cite{bartholdi-v:amenability}, but we can at least show (in a surprisingly difficult manner; Be'eri Greenfeld supplied us with a simpler construction based on~\cite{bell-greenfeld:monomial}):
\begin{mainthm}[= Theorem~\ref{thm:expogrowth}]
  There exists an amenable Lie algebra of exponential growth. It is even (locally finite-dimensional)-by-($1$-dimensional).
\end{mainthm}
In comparison, it is easy to produce amenable groups of exponential growth, for example solvable groups that are not virtually nilpotent. However, solvable Lie algebras have subexponential growth~\cites{lichtman:growth,smith:subexponential}, and the wreath product of an Abelian Lie algebra with an algebra of subexponential growth again has subexponential growth, see~\cite{petrogradsky:growth2} for detailed estimates.

We finally consider amenability under passage to an associated graded object, and prove the following easy result:
\begin{mainthm}[= Theorem~\ref{thm:graded}]\label{mainthm:graded}
  Let $M$ be a Hopf-module coalgebra with its augmentation filtration. If $M$ is amenable, then its associated graded Hopf-module coalgebra $\gr M$ is amenable.
\end{mainthm}

One motivation comes from questions of Vershik concerning the growth of the lower central and dimension series of an amenable group~\cite{vershik:amenability}*{page 326}, see also~\cite{bartholdi-g:lie}. In a more algebraic form, the question asks whether, for the augmentation ideal $\varpi\triangleleft\Bbbk G$, the dimensions of the quotients $\Bbbk G/\varpi^n$ grow subexponentially. Now $\Bbbk G$ is amenable, so Theorem~\ref{mainthm:graded} shows that its associated graded Hopf algebra $\gr{\Bbbk G}=\bigoplus_{n\ge0}\varpi^n/\varpi^{n+1}$ is amenable.

For simplicity, we assume in the next paragraph that $\Bbbk$ has characteristic $0$. If $\Bbbk$ had positive characteristic, the statements and constructions of this article should be adapted to restricted Lie algebras and their restricted universal enveloping algebra.

Quillen's theorem~\cite{quillen:ab} identifies $\gr{\Bbbk G}$ with the universal enveloping algebra of the graded Lie algebra attached to the dimension series
\[
G_n=G\cap(1+\varpi^n),
\qquad
\LieL_\Bbbk=\bigoplus_{n\ge1}(G_n/G_{n+1})\otimes_{\mathbb Z}\Bbbk.
\]
Thus the Lie algebra $\LieL_\Bbbk$ is amenable, and Vershik's question is whether every amenable algebra of this form has subexponential growth.

\subsection*{Sketch of the main argument}
The heart of the paper is the quantitative rounding theorem, Theorem~\ref{thm:hopf}. Given a finite-dimensional subcoalgebra $F\leq H$ and a finite-dimensional subspace $E\leq M$, put $A=\Bbbk1+F$ and choose a finite-dimensional subcoalgebra $G\leq M$ containing $E$. For $t\geq0$, let $C_t$ be the largest subcoalgebra of $G$ maximizing
\[
 \dim(E\cap C)-t\dim C.
\]
The density/coarea identity says that $\int_0^\infty\dim C_t\,dt=\dim E$, while maximality gives a semistability inequality for $U_t=E\cap C_t$ relative to every subcoalgebra of $C_t$.

The main algebraic input is Proposition~\ref{prop:module-coalgebra-transfer}. It says that this semistability inequality survives the action of the finite-dimensional cocommutative coalgebra $A$. When $A$ has a complete flag by subcoalgebras, the proof proceeds one codimension-one step at a time: the new one-dimensional layer produces a group-like element, and the resulting dimension estimates telescope. Over an arbitrary field, one passes to a finite extension where such a flag exists, using tensor amplification to preserve semistability under base change.

Let $D_t$ be the corresponding density maximizer for $AE$. The transfer inequality and maximality imply the pointwise inclusion $AC_t\leq D_t$. Hence
\[
 \int_0^\infty\dim AC_t\,dt\leq\dim AE,
 \qquad
 \int_0^\infty\dim C_t\,dt=\dim E.
\]
For some $t$ with $C_t\neq0$, therefore, $\dim AC_t/\dim C_t\leq\dim AE/\dim E$. Since $AV=V+FV$, this is exactly the desired rounding estimate. The equivalence between coalgebraic and ordinary module amenability follows by first enlarging an arbitrary finite-dimensional test space in $H$ to a finite-dimensional subcoalgebra.

\subsection*{Acknowledgement of AI use}
The author is grateful to Be'eri Greenfeld and Artem Semidetnov for entertaining discussions, and valuable comments on the clarity of the text.

ChatGPT 5.6 Sol was used to assist with drafting, editing, checking parts of this manuscript, and exploratory formal verification. All mathematical content was reviewed, checked, and substantially revised by the author, who is responsible for the final text.

All the main results have been formalized by ChatGPT in \textsf{Lean/Mathlib}, and are available at\\
\centerline{\url{https://github.com/laurentbartholdi/hopfamenability}}

\subsection*{Table of contents}
Section~\ref{ss:rounding} proves the rounding and quotient results for Hopf-module coalgebras. Section~\ref{ss:hopfamen} studies amenability of Hopf algebras, and proves permanence under Hopf subalgebras in the cocommutative case and under cleft extensions. Section~\ref{ss:quotients} applies these results to subgroups and extensions of amenable groups, and proves Gromov's claim about modules of amenable groups. Sections~\ref{ss:lieamen} and~\ref{ss:elem} develop the ``expected'' results respectively on amenability and elementary amenability of Lie algebras. Section~\ref{ss:expgrowth} gives an example of amenable Lie algebra of exponential growth. Section~\ref{ss:graded} develops the theory for graded Hopf-module coalgebras, with an eye towards Vershik's conjecture.

%%%%%%%%%%%%%%%%%%%%%%%%%%%%%%%%%%%%%%%%%%%%%%%%%%%%%%%%%%%%%%%%
\section{Rounding in Hopf-module coalgebras}\label{ss:rounding}

Let $H$ be a cocommutative Hopf algebra over the field $\Bbbk$; see~\cite{montgomery:haaar} for a classical reference.  Let $M$ be a left $H$-module coalgebra; thus $M$ is a coalgebra,
endowed with a left $H$-module structure, and
\[
 \Delta(hm)
   = \sum h_{(1)}m_{(1)}\otimes h_{(2)}m_{(2)},
 \qquad
 \varepsilon(hm)=\varepsilon(h)\varepsilon(m)
\]
for all $h\in H$ and $m\in M$.

If $A\leq H$ and $V\leq M$ are subspaces, we write
\[
 AV=\langle a v:a\in A,\ v\in V\rangle_\Bbbk.
\]

\begin{thm}[Coalgebraic rounding for module coalgebras]
\label{thm:module-coalgebra-rounding}
Let $F\leq H$ be a finite-dimensional subcoalgebra and let
$0\neq E\leq M$ be a finite-dimensional subspace.  Then there exists
a nonzero finite-dimensional subcoalgebra $C\leq M$ such that
\[
 \frac{\dim(C+FC)}{\dim C}
 \leq
 \frac{\dim(E+FE)}{\dim E}.
\]
\end{thm}

This is really a strengthened form of Theorem~\ref{thm:hopf}. Indeed, given a finite-dimensional subspace $F\le H$ and $\epsilon>0$, the fundamental theorem of coalgebras~\cite{sweedler:ha}*{Theorem~2.2.1} provides a finite-dimensional subcoalgebra $F_1\le H$ containing $F$. If $M$ is amenable qua ordinary $H$-module, choose a finite-dimensional subspace $E\le M$ with $\dim(E+F_1E)<(1+\epsilon)\dim E$. Theorem~\ref{thm:module-coalgebra-rounding}, applied to $F_1$, produces a finite-dimensional subcoalgebra $C\le M$ with
\[
\dim(C+F_1C)<(1+\epsilon)\dim C,
\]
and therefore also $\dim(C+FC)<(1+\epsilon)\dim C$. Thus $M$ is amenable qua Hopf-module coalgebra. The converse direction is immediate.

The proof is based on a density filtration. Even though our treatment is entirely self-contained, the ideas originate in~\cite{harder-narasimhan:cohomology} in filtrations of vector bundles, and in combinatorial optimization~\cites{fujishige:psss,fujishige:submodular}. We took inspiration from~\cite{randriambololona:hn}'s lattice-theoretic presentation of the Harder-Narasimhan theory. We shall use the following elementary facts about density filtrations repeatedly.

\begin{lem}[Density/coarea lemma]\label{lem:density-coarea}
Let $W$ be a finite-dimensional vector space and let $\mathcal L$ be a collection of subspaces of $W$, containing $0,W$ and closed under sums and intersections. Let
\[
r\colon\mathcal L\to\mathbb Z_{\ge0}
\]
be monotone and supermodular (namely, $r(V)+r(V')\le r(V\cap V')+r(V+V')$), with $r(0)=0$. For $t\ge0$, there exists a unique largest maximizer $C_t$ of
\[
\Phi_t(V)=r(V)-t\dim V,
\qquad V\in\mathcal L.
\]
Then
\begin{equation}\label{eq:abstract-coarea}
\int_0^\infty \dim C_t\,dt=r(W).
\end{equation}
\end{lem}
The integral is a finite weighted sum: the function $t\mapsto C_t$ has only finitely many changes of dimension.
\begin{proof}
Because $r$ is supermodular and dimension is modular, the sum and intersection of two maximizers are again maximizers. Choose a maximizer of largest dimension. Its sum with any other maximizer is again a maximizer and cannot have larger dimension, so it contains every maximizer. Thus there is a unique largest maximizer $C_t$.

Put
\[
h(t)=\max_{V\in\mathcal L}\{r(V)-t\dim V\}.
\]
Only finitely many pairs $(\dim V,r(V))$ can occur, so $h$ is the maximum of finitely many affine functions. Since $r$ is monotone, $h(0)=r(W)$, while $h(t)=0$ for all sufficiently large $t$, because $0\in\mathcal L$. Away from the finitely many breakpoints, the maximizing dimension is unique and $h'(t)=-\dim C_t$. Integrating gives \eqref{eq:abstract-coarea}.
\end{proof}

\begin{lem}[Comparison and naturality of density filtrations]
\label{lem:density-comparison}
Let $\mathcal L$ be as in Lemma~\ref{lem:density-coarea}, and let $r,\widetilde r$ satisfy the hypotheses of that lemma on $\mathcal L$. Denote their largest maximizers at parameter $t$ by $C_t,\widetilde C_t$. If $\widetilde r-r$ is monotone, then $C_t\leq\widetilde C_t$ for every $t\geq0$.

Moreover, let $T\colon W\to W'$ be a linear isomorphism carrying $\mathcal L$ onto a collection $\mathcal L'$ of subspaces of $W'$, and let $r'\colon\mathcal L'\to\mathbb Z_{\ge0}$ satisfy the hypotheses of Lemma~\ref{lem:density-coarea} and $r'(T(V))=r(V)$. Then $T(C_t)=C'_t$, where $C'_t$ is the largest maximizer associated with $r'$.
\end{lem}

\begin{proof}
For the first assertion, put $C=C_t$, $\widetilde C=\widetilde C_t$, and $K=C\cap\widetilde C$. Maximality of $C$, monotonicity of $\widetilde r-r$, and supermodularity of $\widetilde r$ give
\[
\begin{split}
 \widetilde r(C+\widetilde C)-\widetilde r(\widetilde C)
 &\geq \widetilde r(C)-\widetilde r(K)\\
 &\geq r(C)-r(K)
 \geq t(\dim C-\dim K).
\end{split}
\]
Since $\dim(C+\widetilde C)-\dim\widetilde C=\dim C-\dim K$, the space $C+\widetilde C$ is also a maximizer for $\widetilde r$. The largest-maximizer property therefore gives $C+\widetilde C=\widetilde C$, and hence $C\leq\widetilde C$.

For the second assertion, $T$ preserves dimensions and identifies the two density functionals. It therefore carries maximizers bijectively to maximizers, and hence carries the largest maximizer to the largest maximizer.
\end{proof}

The fundamental notion we shall use is that of ``semistability'', inspired by semistability of vector bundles. Let $C$ be a finite-dimensional coalgebra, let $U\leq C$ be a subspace, and consider $t\in\mathbb R$. The pair $(U,C)$ is \emph{$t$-semistable} if
\begin{equation}\tag{SS}
  t(\dim C-\dim B)\leq \dim U-\dim(U\cap B)
 \label{eq:SS}
\end{equation}
for every subcoalgebra $B\leq C$.

We shall use the following elementary ``tensor amplification'' of the semistability inequality.  It will allow us to pass to an arbitrary finite extension of the ground field.

\begin{lem}[Tensor amplification of semistability]
\label{lem:tensor-semistability}
Let $C$ be a finite-dimensional coalgebra, let $U\leq C$, and let $t\in\mathbb R$.  Suppose that $(U,C)$ is $t$-semistable. Let $W$ be a finite-dimensional vector space, and let $Z\leq W\otimes C$ be a subspace stable under both regular $C$-coactions: for
\[
 \rho(w\otimes c)=\sum (w\otimes c_{(1)})\otimes c_{(2)},
 \qquad
 \lambda(w\otimes c)=\sum c_{(1)}\otimes(w\otimes c_{(2)}),
\]
assume that
\[
 \rho(Z)\subseteq Z\otimes C,
 \qquad
 \lambda(Z)\subseteq C\otimes Z.
\]
Then
\begin{equation}
 t\bigl(\dim W\,\dim C-\dim Z\bigr)
 \leq
 \dim W\,\dim U-\dim\bigl(Z\cap(W\otimes U)\bigr).
 \label{eq:tensor-amplification}
\end{equation}
\end{lem}

\begin{proof}
We argue by induction on $\dim W$.  There is nothing to prove when $W=0$.  Otherwise choose a nonzero functional $\ell\colon W\to\Bbbk$, choose $a\in W$ with $\ell(a)=1$, and put $W_0=\ker\ell$.  Thus $\dim W=\dim W_0+1$. Define
\[
 \phi=\ell\otimes\id_C\colon W\otimes C\rightarrow C,
 \qquad
 B=\phi(Z).
\]
We first claim that $B$ is a subcoalgebra of $C$.  The right-coaction stability of $Z$ gives
\[
 \Delta(B)\subseteq B\otimes C,
\]
because applying $\ell$ to the $W$-factor in $\rho(Z)\subseteq Z\otimes C$ gives precisely the comultiplication of $\phi(Z)$.  Similarly the left-coaction stability gives
\[
 \Delta(B)\subseteq C\otimes B.
\]
Since, over a field,
\[
 (B\otimes C)\cap(C\otimes B)=B\otimes B,
\]
it follows that $B$ is a subcoalgebra.

Put
\[
 Z_0=Z\cap(W_0\otimes C),
\]
viewed as a subspace of $W_0\otimes C$.  Since $W_0\otimes C$ is itself stable under both coactions, so is $Z_0$, and the induction hypothesis applies to $Z_0$.  Moreover $\ker\phi=W_0\otimes C$, so restriction of $\phi$ to $Z$ gives an exact sequence
\[
 0\longrightarrow Z_0\longrightarrow Z\overset{\phi}{\longrightarrow}
 B\longrightarrow0,
\]
and therefore
\begin{equation}
 \dim Z=\dim Z_0+\dim B.
 \label{eq:tensor-Z-dim}
\end{equation}

Set
\[
 N=Z\cap(W\otimes U),
 \qquad
 N_0=Z_0\cap(W_0\otimes U).
\]
The kernel of $\phi|_N$ is $N_0$, while $\phi(N)\subseteq U\cap B$.  Therefore
\begin{equation}
 \dim N\leq \dim N_0+\dim(U\cap B).
 \label{eq:tensor-N-dim}
\end{equation}
By induction,
\[
 t\bigl(\dim W_0\,\dim C-\dim Z_0\bigr)
 \leq
 \dim W_0\,\dim U-\dim N_0,
\]
and by \eqref{eq:SS}, applied to the subcoalgebra $B$,
\[
 t(\dim C-\dim B)\leq \dim U-\dim(U\cap B).
\]
Adding these inequalities, and using~\eqref{eq:tensor-Z-dim}, $\dim W=\dim W_0+1$, and~\eqref{eq:tensor-N-dim}, gives \eqref{eq:tensor-amplification}.
\end{proof}

\subsection{Finite extensions of the ground field}
We isolate three elementary scalar-extension facts.

\begin{lem}[Base change of semistability]\label{lem:semistability-base-change}
Let $\K/\Bbbk$ be a finite field extension, let $C$ be a finite-dimensional coalgebra, and let $U\le C$. If $(U,C)$ is $t$-semistable for some $t\in\mathbb R$,
then $(\K\otimes U,\K\otimes C)$ is $t$-semistable.
\end{lem}

\begin{proof}
Write $C_\K=\K\otimes_\Bbbk C$ and $U_\K=\K\otimes_\Bbbk U$, put $d=[\K:\Bbbk]$. Let $Z\le C_\K$ be a $\K$-subcoalgebra, regarded as a $\Bbbk$-subspace of $C_\K$. Since $Z$ is a $\K$-subcoalgebra, after restricting scalars it is stable under both regular $C$-coactions. Indeed, applying to $\Delta_\K(Z)\subseteq Z\otimes_\K Z$ the canonical cancellation maps
\[
 C_\K\otimes_\K C_\K\longrightarrow C_\K\otimes_\Bbbk C,
 \qquad
 C_\K\otimes_\K C_\K\longrightarrow C\otimes_\Bbbk C_\K
\]
gives respectively the coactions $\alpha\otimes c\mapsto\sum(\alpha\otimes c_{(1)})\otimes c_{(2)}$ and $\alpha\otimes c\mapsto\sum c_{(1)}\otimes(\alpha\otimes c_{(2)})$. Lemma~\ref{lem:tensor-semistability}, with $W=\K$, therefore gives
\[
 t(d\dim_\Bbbk C-\dim_\Bbbk Z)
 \le d\dim_\Bbbk U-\dim_\Bbbk\bigl(Z\cap(\K\otimes U)\bigr).
\]
All spaces involving $Z$ are $\K$-subspaces, so division by $d$ gives the result.
\end{proof}

\begin{lem}[Splitting a finite cocommutative coalgebra]\label{lem:coalgebra-splitting-field}
  Let $A$ be a finite-dimensional cocommutative coalgebra over $\Bbbk$. There is a finite field extension $\K/\Bbbk$ such that
  \[A_\K\coloneqq\K\otimes_\Bbbk A\]
  admits a complete flag by subcoalgebras.
\end{lem}

\begin{proof}
Choose a basis $(e_i)$ of $A$, with dual basis $(e_i^*)$, and consider the coefficient operators $T_i=(\id\otimes e_i^*)\Delta$. Coassociativity and cocommutativity imply that the $T_i$ commute. Choose a finite extension $\K/\Bbbk$ over which all their characteristic polynomials split. The $T_i$ are then simultaneously triangularizable on $A_\K$, so they have a common invariant complete flag. Invariance under all coefficient operators is exactly the condition $\Delta(V)\subseteq V\otimes A_\K$; by cocommutativity it also gives $\Delta(V)\subseteq A_\K\otimes V$. Hence every member of the flag is a subcoalgebra.
\end{proof}

\begin{lem}[Descent of Følner ratios]\label{lem:folner-descent}
Let $\K/\Bbbk$ be finite, let $A$ be a finite-dimensional subspace of a $\Bbbk$-algebra $H$, and let $M$ be a left $H$-module. Write $A_\K=\K\otimes A$ and $M_\K=\K\otimes M$. For every nonzero finite-dimensional $\K$-subspace $V\le M_\K$ there is a nonzero finite-dimensional $\Bbbk$-subspace $V_0\le M$ such that
\[
 \frac{\dim_\Bbbk AV_0}{\dim_\Bbbk V_0}
 \le \frac{\dim_\K A_\K V}{\dim_\K V}.
\]
\end{lem}

\begin{proof}
Choose a $\Bbbk$-basis $\alpha_1,\ldots,\alpha_d$ of $\K$ and put $K_i=\langle\alpha_1,\ldots,\alpha_i\rangle_\Bbbk$. Filter $V$ by $V^{(i)}=V\cap(K_i\otimes M)$. The coefficient of $\alpha_i$ identifies $V^{(i)}/V^{(i-1)}$ with a subspace $V_i\le M$. If $W=A_\K V$ and $W_i\le M$ are defined similarly, then $AV_i\subseteq W_i$. Moreover
\[
 \sum_i\dim V_i=d\dim_\K V,
 \qquad
 \sum_i\dim W_i=d\dim_\K W.
\]
Thus for some $i$ with $V_i\ne0$ we have $\dim W_i/\dim V_i\le\dim_\K W/\dim_\K V$. Taking $V_0=V_i$ proves the claim.
\end{proof}

\subsection{Transfer}
We next record the transfer statement needed in the argument.

\begin{prop}[Module-coalgebra transfer]
\label{prop:module-coalgebra-transfer}
Let $A\leq H$ and $C\leq M$ be finite-dimensional subcoalgebras, with $A$ cocommutative.  Consider $U\leq C$ and $t\in\mathbb R$, and suppose that $(U,C)$ is $t$-semistable. Then $(AU,AC)$ is $t$-semistable.
\end{prop}

\begin{proof}
We first prove the assertion under the additional assumption that $A$ admits a complete flag by subcoalgebras
\[
 0=A_n<A_{n-1}<\cdots<A_1<A_0=A,
 \qquad
 \dim A_i=\dim A_{i+1}+1.
\]

Consider first a single codimension-one step
\[
 A'\leq A,
 \qquad
 \dim A=\dim A'+1.
\]
Choose a linear functional $\ell\colon A\to \Bbbk$ and $a\in A$ such that
\[
 \ker\ell=A',
 \qquad
 \ell(a)=1.
\]
Set
\[
 g=(\ell\otimes\id)\Delta(a).
\]
The usual codimension-one coalgebra argument gives
\begin{equation} \label{eq:delta-a-right}
 \Delta(a)-a\otimes g\in A'\otimes A,
\end{equation}
and $g$ is group-like. Indeed consider $T=(\ell\otimes\id)\Delta\colon A\to A$; then $T(A')=0$, so for $x\in A$ we have $x-\ell(x)a\in A'$ and therefore $T(x)=T(x-\ell(x)a)+T(\ell(x)a)=\ell(x)g$. We then verify
\[\varepsilon(g)=(\ell\otimes\varepsilon)\Delta(a)=\ell(a)=1\]
and
\begin{align*}
\Delta(g)
 &= (\ell\otimes\id\otimes\id)(\id\otimes\Delta)\Delta(a)\\
 &= (\ell\otimes\id\otimes\id)(\Delta\otimes\id)\Delta(a)
  = (T\otimes\id)\Delta(a)=g\otimes g.
\end{align*}

Since $A$ is cocommutative, we also have
\begin{equation}\label{eq:delta-a-left}
 \Delta(a)-g\otimes a\in A\otimes A'.
\end{equation}
Moreover, the action of $g$ on $M$ is injective, since $S(g)$ is a left inverse to the action of $g$.

Fix a subcoalgebra $D\leq AC$, and put
\[
 V=D+A'C.
\]
Since the action $A\otimes C\to M$ is a coalgebra map, $A'C$ is a subcoalgebra of $M$, and therefore so is $V$.

Define
\[
 \phi\colon C\to AC/V,
 \qquad
 c\mapsto ac+V.
\]
This map is surjective, because $A=A'+\Bbbk a$.  Put
\[
 B=\ker\phi.
\]
We claim that $B$ is a subcoalgebra of $C$. Indeed, consider $c\in B$. Then $ac\in V$.  If
\[
 q\colon AC\rightarrow AC/V
\]
is the quotient map, then $(q\otimes\id)\Delta(ac)=0$. Using the module-coalgebra identity and~\eqref{eq:delta-a-right}, all terms whose first $A$-factor lies in $A'$ disappear modulo $V$, and we obtain
\[
 \sum \phi(c_{(1)})\otimes gc_{(2)}=0.
\]
Since multiplication by the group-like element $g$ is injective on $M$, and tensoring over a field preserves injections, $(\phi\otimes\id)\Delta(c)=0$, so $\Delta(c)\in B\otimes C$. Applying instead $\id\otimes q$ to $\Delta(ac)$, and using~\eqref{eq:delta-a-left}, gives similarly $\Delta(c)\in C\otimes B$. Over a field,
\[
 (B\otimes C)\cap(C\otimes B)=B\otimes B,
\]
inside $C\otimes C$.  Thus $\Delta(B)\subseteq B\otimes B$, so $B$ is a subcoalgebra.

By the rank-nullity formula for $\phi$,
\begin{equation} \label{eq:module-step-denominator}
 \dim AC-\dim(D+A'C)
 =
 \dim C-\dim B.
\end{equation}

Now consider
\[
 \psi\colon U\to (D+AU)/(D+A'U),
 \qquad
 u\mapsto au.
\]
Again this is surjective.  If $u\in\ker\psi$, then
\[
 au\in D+A'U\subseteq D+A'C,
\]
and hence $u\in B$.  Therefore $\ker\psi\leq U\cap B$, and consequently
\begin{equation} \label{eq:module-step-numerator}
 \dim(D+AU)-\dim(D+A'U)
 \geq
 \dim U-\dim(U\cap B).
\end{equation}
Combining \eqref{eq:SS}, \eqref{eq:module-step-denominator}, and~\eqref{eq:module-step-numerator}, we obtain the codimension-one
estimate
\begin{equation} \label{eq:module-codim-one}
 t\bigl(\dim AC-\dim(D+A'C)\bigr)
 \leq
 \dim(D+AU)-\dim(D+A'U).
\end{equation}

We now iterate this estimate along a complete subcoalgebra flag in $A$.  If $A'<A$ is the first proper member of the flag, apply~\eqref{eq:module-codim-one} to $A'<A$, and apply the induction hypothesis to
\[
 D'=D\cap A'C.
\]
The two left-hand sides telescope because
\[
 \begin{aligned}
 &\bigl[\dim AC-\dim(D+A'C)\bigr]
   +\bigl[\dim A'C-\dim(D\cap A'C)\bigr] \\
 &\hspace{35mm}=\dim AC-\dim D,
 \end{aligned}
\]
by the modular dimension identity.  Likewise, since $A'U\leq A'C$,
\[
 A'U\cap(D\cap A'C)=A'U\cap D,
\]
and
\[
 \begin{split}
 \bigl[\dim(D+AU)-\dim(D+A'U)\bigr]+\bigl[\dim A'U-\dim(A'U\cap D)\bigr]\\
   =\dim AU-\dim(AU\cap D).
 \end{split}
\]
This proves that $(AU,AC)$ is $t$-semistable whenever $A$ admits a complete subcoalgebra flag.

It remains to remove the complete-flag assumption. Choose a finite extension $\K/\Bbbk$ as in Lemma~\ref{lem:coalgebra-splitting-field}, so that $A_\K$ has a complete flag by subcoalgebras. By Lemma~\ref{lem:semistability-base-change}, the semistability hypothesis \eqref{eq:SS} holds for $C_\K$ and $U_\K$. We may therefore apply the complete-flag case over $\K$ to the scalar extensions of $A,C,U,D$. Scalar extension preserves the relevant dimensions, products and intersections (for instance $(AU\cap D)_\K=A_\K U_\K\cap D_\K$), so the resulting inequality over $\K$ is exactly $t$-semistability of $(AU,AC)$ over $\Bbbk$.
\end{proof}

\begin{proof}[Proof of Theorem~\ref{thm:module-coalgebra-rounding}]
Put $A=\Bbbk1+F$. Then $A$ is a finite-dimensional cocommutative subcoalgebra of $H$, and $AV=V+FV$ for every subspace $V\leq M$.

By the fundamental theorem of coalgebras, there exists a finite-dimensional subcoalgebra $G\leq M$ containing $E$. For $t\ge0$, let $C_t=C_t(E)\leq G$ be the largest subcoalgebra maximizing
\[
 \Phi_{E,t}(C)\coloneqq\dim(E\cap C)-t\dim C.
\]
The largest maximizer exists because the function $C\mapsto\dim(E\cap C)$ is supermodular, whereas dimension is modular, and subcoalgebras are closed under sums and intersections so Lemma~\ref{lem:density-coarea} applies.

Put $U_t=E\cap C_t$. The defining maximality of $C_t$ gives the semistability inequality
\begin{equation} \label{eq:density-semistability}
 t\bigl(\dim C_t-\dim B\bigr)
 \leq
 \dim U_t-\dim(U_t\cap B)
\end{equation}
for every subcoalgebra $B\leq C_t$.

Now consider the corresponding density filtration for $AE$ inside $AG\leq M$. Let $D_t\coloneqq C_t(AE)$ denote its largest maximizer at parameter $t$. We claim that
\begin{equation} \label{eq:density-functoriality}
 AC_t\leq D_t.
\end{equation}
Indeed, put $C'=AC_t$ and $U'=AU_t$.  By Proposition~\ref{prop:module-coalgebra-transfer}, the pair $(U',C')$ is $t$-semistable.  Assume first that $t>0$, and on the lattice of subcoalgebras of $AG$ consider the two rank functions
\[
 r_0(B)=\dim(U'\cap B),
 \qquad
 r_1(B)=\dim(AE\cap B).
\]
The largest maximizer associated with $r_0$ is $C'$.  Indeed, for any subcoalgebra $B\leq AG$, replacing $B$ by $B\cap C'$ leaves $r_0(B)$ unchanged and, unless $B\leq C'$, strictly decreases its dimension; while $t$-semistability says precisely that
\[
 r_0(B)-t\dim B\leq r_0(C')-t\dim C'
 \qquad(B\leq C').
\]
Moreover, $r_1-r_0$ is monotone, since
\[
 r_1(B)-r_0(B)
 =\dim\operatorname{im}(AE\cap B\longrightarrow AE/U'),
\]
and these images increase with $B$.  Lemma~\ref{lem:density-comparison} therefore gives $C'\leq D_t$.  For $t=0$ the same conclusion is immediate, since $D_0=AG$.  This proves \eqref{eq:density-functoriality}.

It remains only to average over the density filtration. Apply Lemma~\ref{lem:density-coarea} to the lattice of subcoalgebras of $G$, with
\[
r(C)=\dim(E\cap C).
\]
Since $E\le G$, the lemma gives
\begin{equation}\label{eq:source-mass}
\int_0^\infty\dim C_t(E)\,dt=\dim E.
\end{equation}
Applying the same lemma to $AE\le AG$ gives
\[
\int_0^\infty\dim D_t\,dt=\dim AE.
\]
By \eqref{eq:density-functoriality}, we have $\dim AC_t\leq\dim D_t$,
and therefore
\begin{equation}\label{eq:expanded-mass}
\int_0^\infty\dim AC_t\,dt\leq\dim AE.
\end{equation}
Together with \eqref{eq:source-mass}, this implies that for some $t$ with $C_t\neq0$,
\[
\frac{\dim AC_t}{\dim C_t}
\leq
\frac{\dim AE}{\dim E}.
\]
Indeed, otherwise the integrand on the left of \eqref{eq:expanded-mass} would be everywhere strictly larger, on the support of $\dim C_t$, than $\frac{\dim AE}{\dim E}\dim C_t$, contradicting \eqref{eq:source-mass} and \eqref{eq:expanded-mass}.

Taking $C=C_t$ and recalling $AC=C+FC$ and $AE=E+FE$, we obtain
\[
 \frac{\dim(C+FC)}{\dim C}
 \leq
 \frac{\dim(E+FE)}{\dim E},
\]
as required.
\end{proof}

%---------------------------------------------------------------
\subsection{Quotients of Hopf-module coalgebras} We apply similar methods to prove that every coalgebra quotient of an amenable module is amenable. Note that the quotient should be counital, so cannot be the $0$ module --- and indeed, the $0$ module is not amenable.

\begin{thm}\label{thm:module-coalgebra-quotient}
Let $H$ be a cocommutative Hopf algebra and let $q\colon M\twoheadrightarrow Q$ be a surjective $H$-linear counital coalgebra morphism between $H$-module coalgebras. If $M$ is amenable, then $Q$ is amenable.
\end{thm}

\begin{proof}
By Theorem~\ref{thm:hopf}, it suffices to prove ordinary algebraic amenability of $Q$. Let $F_0\le H$ be finite-dimensional and $\epsilon>0$. Choose a finite-dimensional subcoalgebra $F\le H$ containing $F_0+\Bbbk1$. Since $M$ is amenable, choose a finite-dimensional $E\le M$ with $\dim FE<(1+\epsilon)\dim E$. We shall produce a nonzero finite-dimensional $V\le Q$ with
\[
 \frac{\dim FV}{\dim V}\le\frac{\dim FE}{\dim E}.
\]
This is enough because $1\in F$ and $F_0\le F$.

Consider the right $Q$-coaction $\rho=(\id_M\otimes q)\Delta_M\colon M\to M\otimes Q$. It is injective, since $(\id_M\otimes\varepsilon_Q)\rho=\id_M$, and it is $H$-linear for the diagonal action. Put $X=\rho(E)$ and $Y=\rho(FE)$; thus $X\le Y$, $\dim X=\dim E$ and $\dim Y=\dim FE$.

Assume first that $F$ admits a complete flag $0=F'_0<F'_1<\cdots<F'_d=F$ by subcoalgebras, with $\dim F'_i=i$. Choose a finite-dimensional subspace $W\le Q$ containing all second coefficients occurring in $X$ and $Y$. For a subspace $V\le Q$ set
\[
 r_Z(V)=\dim\bigl(Z\cap(M\otimes V)\bigr),\qquad Z=X,Y.
\]
Replacing $V$ by $V\cap W$ leaves $r_Z(V)$ unchanged and cannot increase its dimension, so every maximizer of $r_Z(V)-t\dim V$ lies in $W$. The identities
\begin{align*}
  (Z\cap (M\otimes V))\cap (Z\cap(M\otimes V')) &= Z\cap(M\otimes(V\cap V')),\\
  (Z\cap (M\otimes V)) + (Z\cap(M\otimes V')) &\le Z\cap(M\otimes(V+V'))
\end{align*}
show that $r_Z$ is supermodular. For $t>0$, let $C_t(Z)$ be the largest maximizer.

We claim that $F C_t(X)\le C_t(Y)$. Fix $t$ and write $C=C_t(X)$ and $D=C_t(Y)$. We prove by induction that $F'_iC\le D$. Suppose $F'_{i-1}C\le D$ and choose $a\in F'_i\setminus F'_{i-1}$. The codimension-one coalgebra argument used in the proof of Proposition~\ref{prop:module-coalgebra-transfer} gives a group-like $g\in F'_i$ such that
\[
 \Delta(a)-a\otimes g\in F'_{i-1}\otimes F'_i,
 \qquad
 \Delta(a)-g\otimes a\in F'_i\otimes F'_{i-1}.
\]
Put $K=\{c\in C:ac\in D\}$, a vector subspace of $C$. Set $D'=D+aC=D+F'_iC$ and let $\delta_a$ denote the diagonal action of $a$ on $M\otimes Q$. Since $\rho$ is $H$-linear, $\delta_a(X)\subseteq Y$, and modulo $M\otimes D$ the second triangular relation gives $\delta_a z\equiv(L_g\otimes L_a)z$ for $z\in M\otimes C$. As $L_g$ is invertible, $\delta_a$ induces an injection
\[
 \frac{X\cap(M\otimes C)}{X\cap(M\otimes K)}
 \hookrightarrow
 \frac{Y\cap(M\otimes D')}{Y\cap(M\otimes D)}.
\]
Hence $r_Y(D')-r_Y(D)\ge r_X(C)-r_X(K)\ge t(\dim C-\dim K)$, the last inequality coming from maximality of $C$. Rank--nullity for $C\to Q/D$, $c\mapsto ac+D$, gives $\dim D'-\dim D=\dim C-\dim K$. Therefore $r_Y(D')-t\dim D'\ge r_Y(D)-t\dim D$. Since $D$ is the largest maximizer, $D'\le D$, and hence $F'_iC\le D$. This proves the claim.

Apply Lemma~\ref{lem:density-coarea} to the lattice of subspaces of $W$. Since $r_Z(W)=\dim Z$,
\[
 \int_0^\infty\dim C_t(X)\,dt=\dim X,
 \qquad
 \int_0^\infty\dim C_t(Y)\,dt=\dim Y.
\]
The pointwise inclusion gives $\dim F C_t(X)\le\dim C_t(Y)$. Consequently, for some $t>0$ with $C_t(X)\ne0$,
\[
 \frac{\dim F C_t(X)}{\dim C_t(X)}
 \le\frac{\dim Y}{\dim X}
 =\frac{\dim FE}{\dim E}.
\]
This proves the assertion when $F$ has a complete subcoalgebra flag.

For arbitrary $\Bbbk$, choose a finite extension $\K/\Bbbk$ by Lemma~\ref{lem:coalgebra-splitting-field}. Apply the split case to the scalar extensions of $q$ and $E$; it gives a nonzero finite-dimensional $\K$-subspace $V\le Q_\K$ such that
\[
 \frac{\dim_\K F_\K V}{\dim_\K V}
 \le\frac{\dim FE}{\dim E}.
\]
Lemma~\ref{lem:folner-descent} then produces a nonzero finite-dimensional $V_0\le Q$ with no larger $F$-expansion ratio. Thus $Q$ is algebraically amenable, and Theorem~\ref{thm:hopf} finishes the proof.
\end{proof}

%---------------------------------------------------------------
\section{Amenability of Hopf algebras}\label{ss:hopfamen}

By Theorem~\ref{thm:hopf}, a cocommutative Hopf algebra is amenable if and only if its left regular module is algebraically amenable. We first record the module-theoretic descent argument needed below.

\subsection{Elementary operations on Hopf algebras}

\begin{lem}[Projective-module descent]\label{lem:projective-module-descent}
Let $R$ be an associative algebra and let $Q$ be a projective left $R$-module. If $Q$ is algebraically amenable qua left $R$-module, then $R$ is algebraically amenable.
\end{lem}

\begin{proof}
Let $F\leq R$ be finite-dimensional and let $0\neq E\leq Q$ be finite-dimensional. Since $Q$ is projective, there is an injective left $R$-linear map
\[
\iota\colon Q\longrightarrow R^{(I)}
\]
for some index set $I$. The finite-dimensional space $\iota(E+FE)$ is supported on finitely many coordinates. After renumbering them, put
\[
W=\iota(E)\leq R^r,
\qquad
W^+=\iota(E+FE)=W+FW\leq R^r.
\]
Set
\[
P_j=Re_1\oplus\cdots\oplus Re_j
\qquad(0\leq j\leq r),
\]
and, using $P_j/P_{j-1}\cong R$ qua left $R$-modules, regard
\[
V_j=\frac{(W\cap P_j)+P_{j-1}}{P_{j-1}},
\qquad
V_j^+=\frac{(W^+\cap P_j)+P_{j-1}}{P_{j-1}}
\]
as finite-dimensional subspaces of $R$. Since the $P_j$ are left $R$-submodules, we have $V_j+FV_j\subseteq V_j^+$. Moreover,
\[
\dim E=\dim W=\sum_{j=1}^r\dim V_j,
\qquad
\dim(E+FE)=\dim W^+=\sum_{j=1}^r\dim V_j^+.
\]
Consequently, for some $j$ with $V_j\neq0$,
\[
\frac{\dim(V_j+FV_j)}{\dim V_j}
\leq
\frac{\dim(E+FE)}{\dim E}.
\]
Applying this to arbitrarily good Følner subspaces $E\leq Q$ proves the lemma.
\end{proof}

\begin{thm}[Hopf subalgebras]\label{thm:hopf-subalgebra}
Let $K\leq H$ be a Hopf subalgebra of a cocommutative Hopf algebra $H$. If $H$ is amenable, then $K$ is amenable.
\end{thm}
\begin{proof}
Takeuchi proved that $H$ is faithfully flat qua left and right $K$-module~\cite{takeuchi:correspondence}*{Theorem~3.1}. The antipodes of $H$ and $K$ are involutive, since these Hopf algebras are cocommutative. The Masuoka--Wigner theorem~\cite{masuoka-wigner:flatness}*{Theorem~2.1} therefore implies that ${}_K H$ is a projective generator, and in particular a projective left $K$-module.

By Theorem~\ref{thm:hopf}, the left regular $H$-module is algebraically amenable. Restricting its action to $K$ shows that $H$ is algebraically amenable qua left $K$-module. Lemma~\ref{lem:projective-module-descent} then shows that $K$ is algebraically amenable, and Theorem~\ref{thm:hopf} shows that $K$ is amenable.
\end{proof}

\begin{thm}[Unions]\label{thm:union}
  Let $H=\bigcup_{i\in I}H_i$ be a directed union of amenable Hopf algebras. Then $H$ is amenable.
\end{thm}
\begin{proof}
  For $F\le H$ finite dimensional and $\epsilon>0$, let $i\in I$ be such that $F\le H_i$. Since $H_i$ is amenable, there exists a finite-dimensional subcoalgebra $E\le H_i$ with $\dim(E+F E)<(1+\epsilon)\dim E$. Viewing $E$ as a subcoalgebra of $H$ shows that $H$ is amenable.
\end{proof}

For a Hopf algebra $H$ and a finite-dimensional subspace $F\le H$, the corresponding ``balls'' in $H$ are defined by
\[B_0=\Bbbk1\text{ and }B_{n+1}=B_n+F B_n\text{ for }n\ge0.\]
If $H$ is generated by $F$, its \emph{growth} is the function $n\mapsto\dim B_n$, and $H$ has \emph{subexponential growth} if $\limsup(\dim B_n)^{1/n}=1$. In general, $H$ has \emph{locally subexponential growth} if $\limsup(\dim B_n)^{1/n}=1$ for every finite-dimensional $F\le H$, namely every finitely generated subalgebra has subexponential growth. Note that if $F$ is a subcoalgebra then every ball $B_n$ is also a subcoalgebra.

\begin{thm}[Subexponential growth]\label{thm:subexp}
  Let $H$ be a Hopf algebra locally of subexponential growth. Then $H$ is amenable.
\end{thm}
\begin{proof}
  Consider $F\le H$ finite-dimensional and $\epsilon>0$, and let $C\le H$ a finite-dimensional subcoalgebra containing $1$ and $F$. Replace $H$ by the subalgebra generated by $C$. By assumption, it has subexponential growth. Let $B_n$ be its balls. From $\limsup(\dim B_n)^{1/n}=1$ it follows that there exists $n$ with $\dim B_{n+1}/\dim B_n<1+\epsilon$, so $E=B_n$ witnesses $\dim(E+F E)<(1+\epsilon)\dim E$.
\end{proof}
This applies in particular to $H$ a commutative Hopf algebra, since then $\dim B_n\le\binom{\dim C+n-1}{n}$.

\subsection{Extensions of Hopf algebras}
For us, a \emph{cleft exact sequence} means an exact sequence
\[
\Bbbk\longrightarrow A\longrightarrow B
\overset{\pi}{\longrightarrow}C\longrightarrow\Bbbk
\]
of Hopf algebras (namely, $A=B^{\operatorname{co}C}=\{b\in B:(\id\otimes\pi)\Delta(b)=b\otimes1\}$) for which $\pi$ admits a normalized coalgebra section $\sigma$, namely $\pi\sigma=\id_C$ and $\sigma(1)=1$. This normalization entails no loss: for an arbitrary coalgebra section, $\sigma(1)$ is a group-like element lying in the image of $A$, and $c\mapsto\sigma(1)^{-1}\sigma(c)$ is normalized.

It is usual, in the definition of exact sequences of Hopf algebras, to require $\ker(\pi)=B\ker(\varepsilon_A)$, but this follows from the coinvariants condition: there is an isomorphism of coalgebras $\Theta_r\colon C\otimes A\to B$, see~\eqref{eq:Theta} below, defined by $\Theta_r(c\otimes a)=\sigma(c)a$, and
\[\pi(a)=(\varepsilon_B\otimes\id_C)(\id_B\otimes\pi)\Delta(a)=(\varepsilon_B\otimes\id_C)(a\otimes1)=\varepsilon(a)1.\]
Therefore $\pi\Theta_r(c\otimes a)=c\varepsilon(a)$ and $\ker(\pi)=\sigma(C)\ker(\varepsilon_A)=B\ker(\varepsilon_A)$.

\begin{thm}[Cleft extensions]\label{thm:hopf-extension}
Let
\[
\Bbbk\longrightarrow A\longrightarrow B
\overset{\pi}{\longrightarrow}C\longrightarrow\Bbbk
\]
be a cleft exact sequence of cocommutative Hopf algebras. Then $B$ is amenable if and only if $A$ and $C$ are amenable.
\end{thm}

\begin{proof}
Identify $A$ with its image in $B$. By exactness,
\[
A=B^{\operatorname{co}C}
 =\{b\in B:(\id\otimes\pi)\Delta(b)=b\otimes1\},
\]
and, by cocommutativity, also
\[
A={}^{\operatorname{co}C}B
 =\{b\in B:(\pi\otimes\id)\Delta(b)=1\otimes b\}.
\]
Let $\sigma\colon C\to B$ be a normalized coalgebra section of $\pi$, and put
\[
\overline\sigma=S_B\circ\sigma.
\]
Since $\sigma$ is a coalgebra map, $\overline\sigma$ is its convolution inverse:
\[
\sum\sigma(c_{(1)})\overline\sigma(c_{(2)})
 =\sum\overline\sigma(c_{(1)})\sigma(c_{(2)})
 =\varepsilon(c)1.
\]
Define
\[
r(b)=\sum b_{(1)}\overline\sigma\bigl(\pi(b_{(2)})\bigr),
\qquad
\ell(b)=\sum\overline\sigma\bigl(\pi(b_{(1)})\bigr)b_{(2)}.
\]
A direct coaction calculation gives
\[
(\id\otimes\pi)\Delta(r(b))=r(b)\otimes1,
\qquad
(\pi\otimes\id)\Delta(\ell(b))=1\otimes\ell(b),
\]
so $r(b),\ell(b)\in A$. The two normal-basis maps
\begin{equation}\label{eq:Theta}
\begin{aligned}
\Theta_\ell\colon A\otimes C&\longrightarrow B,
& a\otimes c&\longmapsto a\sigma(c),\\
\Theta_r\colon C\otimes A&\longrightarrow B,
& c\otimes a&\longmapsto\sigma(c)a
\end{aligned}
\end{equation}
are isomorphisms, with inverses
\[
\Theta_\ell^{-1}(b)=\sum r(b_{(1)})\otimes\pi(b_{(2)}),
\qquad
\Theta_r^{-1}(b)=\sum\pi(b_{(1)})\otimes\ell(b_{(2)}).
\]
Indeed, the convolution identities show immediately that these formulas are inverse to $\Theta_\ell$ and $\Theta_r$. Since multiplication in $B$ and the maps $A\to B$ and $\sigma$ are coalgebra maps, both normal-basis isomorphisms are coalgebra isomorphisms. In particular, $B$ is free qua left $A$-module.

Suppose first that $B$ is amenable. Since $A$ is a Hopf subalgebra of the cocommutative Hopf algebra $B$, Theorem~\ref{thm:hopf-subalgebra} shows that $A$ is amenable. The quotient map $\pi\colon B\twoheadrightarrow C$, with $C$ regarded as a left $B$-module coalgebra through $\pi$, satisfies the hypotheses of Theorem~\ref{thm:module-coalgebra-quotient}. Thus $C$ is amenable for the $B$-action. Since this action factors through $C$, and every finite-dimensional subspace of $C$ has a finite-dimensional lift to $B$, it follows that $C$ is amenable qua Hopf algebra.

Conversely, suppose that $A$ and $C$ are amenable. Let $F_0\leq B$ be finite-dimensional and consider $\epsilon>0$. Choose $\delta>0$ such that
\[
(1+\delta)^2<1+\epsilon,
\]
and choose a finite-dimensional subcoalgebra $F\leq B$ containing $F_0+\Bbbk1$. Put
\[
\overline F=\pi(F)\leq C.
\]
By amenability of $C$, there is a nonzero finite-dimensional subcoalgebra $P\leq C$ such that, with
\[
P^+=\overline F P,
\]
one has
\begin{equation}\label{eq:cleft-quotient-folner}
\dim P^+<(1+\delta)\dim P.
\end{equation}
Here $P\leq P^+$ because $1\in\overline F$, and $P^+$ is a subcoalgebra because multiplication $C\otimes C\to C$ is a coalgebra map.

Consider the left $C$-coaction
\[
\lambda=(\pi\otimes\id)\Delta_B\colon B\longrightarrow C\otimes B.
\]
Under $\Theta_r$, it is the regular coaction on the first tensor factor:
\begin{equation}\label{eq:cleft-regular-coaction}
(\id_C\otimes\Theta_r^{-1})\lambda\Theta_r
 =\Delta_C\otimes\id_A.
\end{equation}
Indeed,
\[
\lambda(\sigma(c)a)
 =\sum c_{(1)}\otimes\sigma(c_{(2)})a
\qquad(c\in C,\ a\in A).
\]
For $f\in F$ and $p\in P$ we have
\[
\lambda(f\sigma(p))
 =\sum\pi(f_{(1)})p_{(1)}
       \otimes f_{(2)}\sigma(p_{(2)})
 \in P^+\otimes B.
\]
It follows that
\begin{equation}\label{eq:cleft-extension-defect}
F\sigma(P)\subseteq\sigma(P^+)A.
\end{equation}
Indeed, if $x\in F\sigma(P)$ and $y=\Theta_r^{-1}(x)$, then
\[
(\Delta_C\otimes\id_A)(y)
 \in P^+\otimes C\otimes A
\]
by \eqref{eq:cleft-regular-coaction}; applying $\id_C\otimes\varepsilon_C\otimes\id_A$ gives $y\in P^+\otimes A$.

Since $F\sigma(P)$ is finite-dimensional, there is a finite-dimensional subspace $D\leq A$, which we may assume contains $1$, such that
\begin{equation}\label{eq:cleft-finite-defect}
F\sigma(P)\subseteq\sigma(P^+)D.
\end{equation}
By amenability of $A$, choose a nonzero finite-dimensional subcoalgebra $Q\leq A$ with
\begin{equation}\label{eq:cleft-kernel-folner}
\dim DQ<(1+\delta)\dim Q.
\end{equation}
Finally put
\[
E=\sigma(P)Q=\Theta_r(P\otimes Q)\leq B.
\]
Since $\Theta_r$ is a coalgebra isomorphism, $E$ is a nonzero finite-dimensional subcoalgebra and
\[
\dim E=\dim P\,\dim Q.
\]
Using \eqref{eq:cleft-finite-defect},
\[
FE\subseteq\sigma(P^+)DQ,
\]
and hence, by the injectivity of $\Theta_r$ and \eqref{eq:cleft-quotient-folner}--\eqref{eq:cleft-kernel-folner},
\[
\begin{split}
\dim FE
&\leq\dim P^+\,\dim DQ\\
&<(1+\delta)^2\dim P\,\dim Q
 <(1+\epsilon)\dim E.
\end{split}
\]
Since $1\in F$ and $F_0\leq F$, we have $E+F_0E\leq FE$. Thus $E$ is an $(F_0,\epsilon)$-Følner subcoalgebra of $B$, and $B$ is amenable.
\end{proof}

%%%%%%%%%%%%%%%%%%%%%%%%%%%%%%%%%%%%%%%%%%%%%%%%%%%%%%%%%%%%%%%% 
\section{Amenability of groups and their modules}\label{ss:quotients}

For a finite subset $S\subseteq G$, a subset $A$ of a left $G$-set, and a subspace $E$ of a left $\Bbbk G$-module, we write
\[
SA=\{sa:s\in S,\ a\in A\},
\qquad
SE=\sum_{s\in S}sE.
\]
Recall that a left $G$-set $X$ is \emph{amenable} if, for every finite $S\subseteq G$ and every $\epsilon>0$, there is a finite subset $A\subseteq X$ with $\#(A\cup SA)<(1+\epsilon)\#A$.

Finite-dimensional subcoalgebras of $\Bbbk G$, respectively $\Bbbk X$, are precisely those of the form $\Bbbk S$, respectively $\Bbbk A$, for finite subsets $S\subseteq G$ and $A\subseteq X$. It follows from Theorem~\ref{thm:module-coalgebra-rounding} that $X$ is amenable if and only if $\Bbbk X$ is amenable. In particular,
\begin{equation}\label{eq:group-ring-amenability}
G\text{ is amenable}
\quad\Longleftrightarrow\quad
\Bbbk G\text{ is an amenable Hopf algebra}.
\end{equation}
We recover the classical definition~\cite{bartholdi:amenability}*{Definition~11.2.2} that $G$ is amenable if and only if every nonempty $G$-set is amenable. We also deduce the following classical result, already appearing in~\cite{vneumann:masses}*{page~79}:

\begin{thm}[Subgroups and extensions]\label{thm:group-permanence}
Every subgroup of an amenable group is amenable, and if
\[
1\longrightarrow N\longrightarrow G
\overset{p}{\longrightarrow}Q\longrightarrow1
\]
is an exact sequence of groups then $G$ is amenable if and only if $N$ and $Q$ are amenable.
\end{thm}

\begin{proof}
Let $K\leq G$. Then $\Bbbk K$ is a Hopf subalgebra of the cocommutative Hopf algebra $\Bbbk G$. If $G$ is amenable, then $\Bbbk G$ is amenable by \eqref{eq:group-ring-amenability}; Theorem~\ref{thm:hopf-subalgebra} shows that $\Bbbk K$ is amenable, and another application of \eqref{eq:group-ring-amenability} shows that $K$ is amenable.

Now consider the displayed exact sequence and choose a set-theoretic section $s\colon Q\to G$ of $p$ with $s(1)=1$. Its linear extension
\[
\sigma\colon\Bbbk Q\longrightarrow\Bbbk G,
\qquad q\longmapsto s(q),
\]
is a normalized coalgebra section of the induced Hopf epimorphism $\Bbbk G\twoheadrightarrow\Bbbk Q$, because group elements are group-like. The coinvariants of this Hopf epimorphism are precisely $\Bbbk N$, so the associated sequence
\[
\Bbbk\longrightarrow\Bbbk N\longrightarrow\Bbbk G
\longrightarrow\Bbbk Q\longrightarrow\Bbbk
\]
is a cleft exact sequence of cocommutative Hopf algebras. Theorem~\ref{thm:hopf-extension}, followed by \eqref{eq:group-ring-amenability}, gives the desired equivalence.
\end{proof}

We next strengthen the module form of the amenability criterion, proving the claimed assertion by Gromov:
\begin{thm}[Quotients of permutation modules]
\label{thm:quotients-of-permutation-modules}
Let $X$ be a left $G$-set and let $q\colon\Bbbk X\twoheadrightarrow V$ be a quotient of left $\Bbbk G$-modules. Suppose that $q(x)\neq0$ for some $x\in X$ whose orbit $G x$ is amenable. Then $V$ is amenable.

In particular, if $G$ is amenable, then every nonzero left $\Bbbk G$-module is amenable.
\end{thm}

The conclusion of the last sentence is claimed in~\cite{bartholdi:amenability}*{Corollary~10.31}, but the proof given there is incorrect. It is actually enlightening to explain the error. For a vector space $V$, the survey considered the abelian group generated by symbols $[W]$, with $W\leq V$ finite-dimensional, subject to the valuation relations
\[
[A]+[B]=[A\cap B]+[A+B].
\]
The intended analogy was with the Boolean lattice of finite subsets. It breaks down because the lattice of subspaces is modular but not distributive, and complements are far from unique. Let $P$ be two-dimensional, fix a line $L<P$, and let $M,N<P$ be two complements of $L$. Then $[L]+[M]=[0]+[P]=[L]+[N]$, so $[M]=[N]$. In fact all lines have the same class: for two distinct lines $L,M$, choose a third line $N<L+M$ and compare the two pairs of complements. If $\lambda$ denotes this common class, a complete flag in a $d$-dimensional subspace $W$ gives
\[
[W]=d\lambda-(d-1)[0].
\]
Thus the class of a finite-dimensional subspace depends only on its dimension. In particular, $[W]=[gW]$ for every linear automorphism $g$, so the claimed unique normal form as a sum of a chain of subspaces is false, and the resulting seminorm cannot detect almost-invariance. A subset of a fixed ambient set has a unique set-theoretic complement, whereas a subspace generally has many linear complements; the valuation relations identify all of them and collapse precisely the information that the argument was meant to retain.

\begin{lem}[Linearization of a finite configuration]
\label{lem:finite-configuration-linearization}
Let $V$ be a vector space, let $S\subseteq\operatorname{GL}(V)$ be finite with $1\in S$, and let $\varnothing\neq A\subseteq V\setminus\{0\}$ be finite. Then there is a nonzero finite-dimensional subspace $E\leq V$ such that
\[
\frac{\dim SE}{\dim E}\leq\frac{\#(SA)}{\#A}.
\]
\end{lem}

\begin{proof}
For a finite set $Z\subseteq V\setminus\{0\}$, put $W_Z=\langle Z\rangle$ and, for $L\leq W_Z$, set $r_Z(L)=\#(Z\cap L)$. This function is monotone and supermodular:
\[
r_Z(L)+r_Z(L')\leq r_Z(L\cap L')+r_Z(L+L'),
\]
because $Z\cap(L\cup L')\subseteq Z\cap(L+L')$. For $t\geq0$, let $C_t(Z)$ be the largest maximizer of $r_Z(L)-t\dim L$. Lemma~\ref{lem:density-coarea} gives
\[
\int_0^\infty\dim C_t(Z)\,dt=\#Z.
\]

We invoke the two functoriality assertions of Lemma~\ref{lem:density-comparison}. If $Z\subseteq Z'$, regard both rank functions on the lattice of subspaces of $W_{Z'}$. For $t>0$, every maximizer for $r_Z$ lies in $W_Z$, since replacing $L$ by $L\cap W_Z$ leaves $r_Z(L)$ unchanged and decreases its dimension unless $L\leq W_Z$. Moreover, $r_{Z'}-r_Z=r_{Z'\setminus Z}$ is monotone. Hence
\[
C_t(Z)\leq C_t(Z').
\]
If $s\in\operatorname{GL}(V)$, naturality gives $C_t(sZ)=sC_t(Z)$.

Apply these facts to $A$. For every $s\in S$, the inclusion $sA\subseteq SA$ gives
\[
sC_t(A)=C_t(sA)\leq C_t(SA),
\]
and therefore $SC_t(A)\leq C_t(SA)$. Consequently
\[
\int_0^\infty\dim\bigl(SC_t(A)\bigr)\,dt\leq\#(SA),
\qquad
\int_0^\infty\dim C_t(A)\,dt=\#A.
\]
Since the integrands are step functions with finitely many values, there is some $t>0$ with $C_t(A)\neq0$ and
\[
\frac{\dim\bigl(SC_t(A)\bigr)}{\dim C_t(A)}
\leq\frac{\#(SA)}{\#A}.
\]
Taking $E=C_t(A)$ proves the lemma.
\end{proof}

\begin{proof}[Proof of Theorem~\ref{thm:quotients-of-permutation-modules}]
Choose $x\in X$ as in the statement and put $v=q(x)$. The map $gx\mapsto gv$ is a surjective morphism of $G$-sets from $Gx$ onto $Gv$. Pushing forward an invariant mean shows that quotients of amenable $G$-sets are amenable, so $Gv$ is amenable. Notice also that $0\notin Gv$, since every group element acts invertibly.

Let $F\leq\Bbbk G$ be finite-dimensional and let $\epsilon>0$. Choose a finite set $S\subseteq G$, with $1\in S$, such that $F\leq\Bbbk S$. Følner's criterion for the amenable $G$-set $Gv$ gives a finite set $A\subseteq Gv$ such that $\#(SA)<(1+\epsilon)\#A$. Lemma~\ref{lem:finite-configuration-linearization} gives a nonzero finite-dimensional $E\leq V$ with
\[
\dim SE\leq\frac{\#(SA)}{\#A}\dim E<(1+\epsilon)\dim E.
\]
Since $1\in S$ and $F\leq\Bbbk S$, one has $E+FE\leq SE$. Hence $\dim(E+FE)<(1+\epsilon)\dim E$, as required.

Finally, suppose that $G$ is amenable and that $0\neq V$ is a left $\Bbbk G$-module. Choose $0\neq v\in V$. The cyclic submodule $\Bbbk Gv$ is a nonzero quotient of the permutation module $\Bbbk G$, so the first part makes it amenable. Its Følner subspaces, viewed inside $V$, show that $V$ is amenable.
\end{proof}

%%%%%%%%%%%%%%%%%%%%%%%%%%%%%%%%%%%%%%%%%%%%%%%%%%%%%%%%%%%%%%%%
\section{Amenability of Lie algebras}\label{ss:lieamen}
We first reconcile our definition, which only tests finite-dimensional subspaces of $\LieL$, with ordinary algebraic amenability of $U(\LieL)$.

\begin{prop}[Generator test]\label{prop:generator-test}
A Lie algebra $\LieL$ is amenable if and only if the left regular $U(\LieL)$-module is algebraically amenable. Equivalently, in the definition of amenability one may use arbitrary finite-dimensional subspaces $E\le U(\LieL)$ and test against arbitrary finite-dimensional subspaces of $U(\LieL)$.
\end{prop}
\begin{proof}
If the left regular $U(\LieL)$-module is algebraically amenable, Theorem~\ref{thm:module-coalgebra-rounding}, applied to $H=M=U(\LieL)$, gives almost-invariant finite-dimensional subcoalgebras, so $\LieL$ is amenable.

Conversely, assume that $\LieL$ is amenable, and let $P\le U(\LieL)$ be finite-dimensional. There exist a finite-dimensional subspace $F\le \LieL$ and an integer $d\ge1$ such that, with
\[
V=\Bbbk1+F,
\]
we have $P\le V^d$. Put $m=\dim V$ and $S_d=1+m+\cdots+m^{d-1}$. Given $\epsilon>0$, choose $\delta>0$ with $\delta S_d<\epsilon$. By amenability of $\LieL$, there is a finite-dimensional subcoalgebra $E\le U(\LieL)$ such that
\[
\dim VE< (1+\delta)\dim E.
\]
Choose a complement $B$ with $VE=E\oplus B$. Then $\dim B<\delta\dim E$, and induction gives
\[
V^dE\subseteq E+B+VB+\cdots+V^{d-1}B.
\]
Hence
\[
\dim(E+PE)\le\dim V^dE
\le \dim E+S_d\dim B
<(1+\epsilon)\dim E.
\]
Thus the left regular $U(\LieL)$-module is algebraically amenable.
\end{proof}

We shall use Proposition~\ref{prop:generator-test} freely below, switching between coalgebra Følner sets and ordinary finite-dimensional Følner subspaces whenever convenient.

For a finitely generated Lie algebra $\LieL=\langle F\rangle$, define the ``balls'' by
\[B_1=F\text{ and }B_{n+1}=B_n+\langle [b,f]:b\in B_n,f\in F\rangle_\Bbbk\text{ for }n\ge1.\]
The \emph{growth} of $\LieL$ is the function $n\mapsto\dim B_n$, and $\LieL$ has \emph{subexponential growth} if $\limsup(\dim B_n)^{1/n}\le1$.
\begin{thm}\label{thm:liesubexp}
  Let $\LieL$ be a Lie algebra locally of subexponential growth (for example, Abelian). Then $\LieL$ is amenable.
\end{thm}
\begin{proof}
  Every finitely generated associative subalgebra of $U(\LieL)$ is contained in $U(\LieM)$ for some finitely generated Lie subalgebra $\LieM\le\LieL$. By hypothesis $\LieM$ has subexponential growth, so $U(\LieM)$ has subexponential growth by Smith's theorem~\cite{smith:subexponential}. Thus $U(\LieL)$ is locally of subexponential growth, so amenable by Theorem~\ref{thm:subexp}, and $\LieL$ is amenable by Proposition~\ref{prop:generator-test}.
\end{proof}

\begin{thm}\label{thm:subalgebra}
  Let $\LieM<\LieL$ be a subalgebra of an amenable Lie algebra $\LieL$. Then $\LieM$ is amenable.
\end{thm}
\begin{proof}
By the PBW theorem, $U(\LieM)$ identifies with a Hopf subalgebra of the cocommutative Hopf algebra $U(\LieL)$. By Theorem~\ref{thm:hopf} and Proposition~\ref{prop:generator-test}, amenability of $\LieL$ is equivalent to amenability of $U(\LieL)$. Theorem~\ref{thm:hopf-subalgebra} therefore makes $U(\LieM)$ amenable, and the same equivalence gives amenability of $\LieM$.
\end{proof}

\begin{thm}\label{thm:extension}
Let
\[
0\longrightarrow \LieK\longrightarrow \LieL
\overset{p}{\longrightarrow}\LieQ\longrightarrow0
\]
be an exact sequence of Lie algebras. Then $\LieL$ is amenable if and only if $\LieK$ and $\LieQ$ are amenable.
\end{thm}
\begin{proof}
Choose a linear section $s\colon\LieQ\to\LieL$ of $p$ and an ordered basis of $\LieQ$. Sending every ordered PBW monomial to the corresponding product of its chosen lifts under $s$ defines a linear section
\[
\sigma\colon U(\LieQ)\longrightarrow U(\LieL)
\]
of the induced Hopf epimorphism. Since the basis elements and their lifts are primitive, the PBW coproduct formula shows that $\sigma$ is a normalized coalgebra map. The PBW theorem also shows that
\[
\Bbbk\longrightarrow U(\LieK)\longrightarrow U(\LieL)
\longrightarrow U(\LieQ)\longrightarrow\Bbbk
\]
is exact, hence cleft. Theorem~\ref{thm:hopf-extension}, together with Theorem~\ref{thm:hopf} and Proposition~\ref{prop:generator-test}, gives the stated equivalence.
\end{proof}

\begin{thm}\label{thm:lieunion}
  Let $\LieL=\bigcup_{i\in I}\LieL_i$ be a directed union of amenable Lie algebras. Then $\LieL$ is amenable.
\end{thm}
\begin{proof}
  By the PBW theorem, $U(\LieL)=\bigcup_{i\in I}U(\LieL_i)$; each $U(\LieL_i)$ is amenable by Theorem~\ref{thm:hopf} and Proposition~\ref{prop:generator-test}, so $U(\LieL)$ is amenable by Theorem~\ref{thm:union}, so $\LieL$ is amenable again by Theorem~\ref{thm:hopf} and Proposition~\ref{prop:generator-test}.
\end{proof}

%%%%%%%%%%%%%%%%%%%%%%%%%%%%%%%%%%%%%%%%%%%%%%%%%%%%%%%%%%%%%%%% 
\section{Elementarily amenable Lie algebras}\label{ss:elem}
Let $\textsf{EL}_0$ denote the class of Abelian and finite-dimensional Lie algebras. For an ordinal $\alpha$, let $\textsf{EL}_{\alpha+1}$ consist of extensions of two Lie algebras in $\textsf{EL}_\alpha$ and directed unions of Lie algebras in $\textsf{EL}_\alpha$. For a limit ordinal $\lambda$, put
\[
\textsf{EL}_\lambda=\bigcup_{\beta<\lambda}\textsf{EL}_\beta.
\]
The hierarchy is increasing, since $0\in\textsf{EL}_\alpha$ and every $\LieL\in\textsf{EL}_\alpha$ is the trivial extension of $\LieL$ by $0$.

We first record that this hierarchy already has the subalgebra and quotient closure built into the definition of elementary amenability used in the introduction.

\begin{lem}
For every ordinal $\alpha$, the class $\textsf{EL}_\alpha$ is closed under taking subalgebras and quotients.
\end{lem}
\begin{proof}
We argue by transfinite induction. The assertion is immediate for $\alpha=0$ and at limit ordinals. Suppose it holds for $\textsf{EL}_\alpha$.

If
\[
0\longrightarrow \LieK\longrightarrow \LieL\overset{p}{\longrightarrow}\LieQ\longrightarrow0
\]
is an extension with $\LieK,\LieQ\in\textsf{EL}_\alpha$ and $\LieM\le \LieL$, then
\[
0\longrightarrow \LieM\cap \LieK\longrightarrow \LieM\longrightarrow p(\LieM)\longrightarrow0
\]
is an extension of two members of $\textsf{EL}_\alpha$. If $\LieJ\triangleleft \LieL$, then $\LieL/\LieJ$ is an extension of
\[
(\LieK+\LieJ)/\LieJ\cong \LieK/(\LieK\cap \LieJ)
\]
by
\[
\LieL/(\LieK+\LieJ)\cong \LieQ/p(\LieJ),
\]
again two members of $\textsf{EL}_\alpha$.

If rather $\LieL=\bigcup_i \LieL_i$ is a directed union with $\LieL_i\in\textsf{EL}_\alpha$, then every subalgebra $\LieM\le \LieL$ is the directed union of the $\LieM\cap \LieL_i$, and every quotient $\LieL/\LieJ$ is the directed union of the images $(\LieL_i+\LieJ)/\LieJ$. By the induction hypothesis these terms lie in $\textsf{EL}_\alpha$. Thus subalgebras and quotients of members of $\textsf{EL}_{\alpha+1}$ again belong to $\textsf{EL}_{\alpha+1}$.
\end{proof}

It follows that $\bigcup_\alpha\textsf{EL}_\alpha$ is precisely the smallest class containing the finite-dimensional and Abelian Lie algebras and closed under subalgebras, quotients, extensions, and directed unions, namely the class \textsf{EL} defined in the introduction.

To separate amenability from elementary amenability, it suffices to exhibit a finitely generated, amenable, infinite-dimensional simple Lie algebra. If we assume $\operatorname{char}\Bbbk=0$, an example is the Witt algebra
\[
\LieW=\bigoplus_{n\in\mathbb Z}\Bbbk e_n,
\qquad
\text{with }[e_i,e_j]=(j-i)e_{i+j}.
\]

We recall for convenience the following standard facts; see~\cite{fuks:infinitedim}:
\begin{lem}\label{lem:witt-properties}
The Witt algebra $\LieW$ is finitely generated, simple, infinite-dimensional, and of linear growth.
\end{lem}
\begin{proof}
The four elements
\[
e_{-2},\ e_{-1},\ e_1,\ e_2
\]
generate $\LieW$. Indeed,
\[
[e_1,e_n]=(n-1)e_{n+1}\qquad(n\ge2)
\]
generates all $e_n$ with $n\ge3$ from $e_1,e_2$, and similarly
\[
[e_{-1},e_n]=(n+1)e_{n-1}\qquad(n\le-2)
\]
generates all $e_n$ with $n\le-3$; finally
\[
[e_{-1},e_1]=2e_0.
\]
With respect to this generating space, every bracket word of length at most $n$ is a linear combination of $e_j$ with $|j|\le2n$. Hence the $n$th ball has dimension at most $4n+1$, so $\LieW$ has linear growth.

To prove simplicity, consider a nonzero $\LieK\triangleleft \LieW$ and choose
\[
0\neq x=\sum_{j=m}^n a_j e_j\in \LieK.
\]
Since
\[
[e_0,e_j]=j e_j,
\]
the operator $\operatorname{ad}(e_0)$ is diagonal on the finite-dimensional span of the support of $x$, with pairwise distinct eigenvalues. A polynomial in $\operatorname{ad}(e_0)$ therefore projects $x$ onto a nonzero scalar multiple of one homogeneous component, so $e_j\in \LieK$ for some $j$. If $j\neq0$, then
\[
[e_j,e_{-j}]=-2j e_0\in \LieK,
\]
while if $j=0$ this is already true. It follows from $[e_0,e_k]=k e_k$ that $e_k\in\LieK$ for every $k\neq0$. Thus $\LieK=\LieW$.
\end{proof}

\begin{cor}\label{thm:amennotelem1}
If $\operatorname{char}\Bbbk=0$, then the Witt algebra $\LieW$ is amenable but not elementarily amenable; and moreover
\[
\textsf{EL}\neq\textsf{SL}.
\]
\end{cor}
\begin{proof}
By Lemma~\ref{lem:witt-properties}, $\LieW$ has subexponential growth, $\LieW\in\textsf{SL}_0$ and it is amenable by Theorem~\ref{thm:liesubexp}.

Suppose for contradiction that $\LieW$ is elementarily amenable, and choose the least ordinal $\alpha$ with $\LieW\in\textsf{EL}_\alpha$. We have $\alpha>0$, since $\LieW$ is neither finite-dimensional nor Abelian, and $\alpha$ is not a limit ordinal by minimality. Write $\alpha=\beta+1$.

If $\LieW$ is a directed union of Lie algebras in $\textsf{EL}_\beta$, then a finite generating set for $\LieW$ lies in one member of the directed system, so that member is all of $\LieW$. This contradicts the minimality of $\alpha$.

Thus $\LieW$ occurs as an extension
\[
0\longrightarrow \LieJ\longrightarrow \LieW\longrightarrow \LieQ\longrightarrow0
\]
with $\LieJ,\LieQ\in\textsf{EL}_\beta$. By simplicity, either $\LieJ=0$ or $\LieJ=\LieW$. In the first case $\LieW\cong \LieQ\in\textsf{EL}_\beta$, and in the second $\LieW=\LieJ\in\textsf{EL}_\beta$, again contradicting minimality. Hence $\LieW$ is not elementarily amenable.
\end{proof}

Assume now that $\operatorname{char}\Bbbk$ is positive. There are examples of \emph{branched} Lie algebras, due to Petrogradsky, Shestakov and Zelmanov~\cite{petrogradsky-s-z:nillie} that are amenable but not elementarily amenable. We briefly recall the construction, as described in~\cite{bartholdi:wreathalg}.

A \emph{self-similar Lie algebra} is a Lie algebra $\LieL$ endowed with an injective homomorphism
\begin{equation}\label{eq:sslie}
  \phi\colon \LieL\to X\otimes\LieL\rtimes\Der(X),
\end{equation}
for a finite-dimensional commutative algebra $X$. Conversely, a choice of generating set and the specification of~\eqref{eq:sslie} on the generators uniquely defines a (maximal) self-similar Lie algebra. 

Let $\Bbbk$ be a field of positive characteristic $p$, and set $X=\Bbbk[x]/(x^p)$. The Petrogradsky-Shestakov-Zelmanov algebra is the Lie algebra generated by $d,v$ and self-similarity map
\[
 \phi(d)=\partial_x,
 \qquad
 \phi(v)=1\otimes d+x^{p-1}\otimes v.
\]
Define the ``first-level subalgebra''
\begin{equation}\label{eq:D}
 \mathfrak D=\ker\bigl(\LieL\overset{\phi}{\longrightarrow}
 (X\otimes\LieL)\rtimes\Der X
 \longrightarrow\Der X\bigr).
\end{equation}
As proved in~\cite{bartholdi:wreathalg}*{Proposition~4.15}, $\LieL$ is $\mathbb R_+$-graded, \emph{recurrent} (meaning that $(\varepsilon\otimes 1)\phi\colon\mathfrak D\to\LieL$ is onto), \emph{transitive} (meaning that the action of $U(\LieL)$ on $X^{\otimes n}$ is cyclic for all $n\in\N$), and \emph{regularly branched} over the finite-codimension ideal
\[
 \LieK=\langle c\rangle^{\LieL},
 \qquad
 c=[v,[d,v]],
\]
meaning $\phi(\LieK)$ contains $X\otimes\LieK$.

\begin{thm}\label{thm:amennotelem2}
For every prime $p$, the Lie algebra $\LieL$ belongs to $\textsf{SL}$ but not to $\textsf{EL}$.
\end{thm}
\begin{proof}
The grading of $\LieL$ has dilation $\lambda>1$ and positive-degree generators. It follows from~\cite{bartholdi:wreathalg}*{Proposition~2.17} that $\LieL$ has finite Gelfand-Kirillov dimension, hence subexponential growth. Thus $\LieL\in\textsf{SL}_0$ and $\LieL$ is amenable by Theorem~\ref{thm:liesubexp}.

Suppose, for a contradiction, that $\LieL$ is elementarily amenable, and choose the least ordinal $\alpha$ with $\LieL\in\textsf{EL}_\alpha$. Since $\LieL$ is infinite-dimensional and nonabelian, $\alpha>0$; by minimality it is not a limit ordinal. Thus $\alpha=\beta+1$ for some $\beta$. Since $\LieL$ is finitely generated, it cannot be a proper directed union of members of $\textsf{EL}_\beta$; hence it occurs in an exact sequence
\[
 0\longrightarrow\LieJ\longrightarrow\LieL\longrightarrow\LieQ\longrightarrow0
\]
with $\LieJ,\LieQ\in\textsf{EL}_\beta$. Neither $\LieJ=0$ nor $\LieJ=\LieL$ is possible, since either case would already put $\LieL$ in $\textsf{EL}_\beta$. In particular, $\LieJ$ is nonzero.

We first show that $\LieK$ occurs as a subquotient---in fact as a subalgebra---of every nonzero ideal of $\LieL$. By~\cite{bartholdi:wreathalg}*{Proposition~3.8}, for some $n$ one has, under the iterated self-similarity embedding,
\[
 X^{\otimes n}\otimes[\LieK,\LieK]\subseteq\phi^n(\LieJ).
\]
Regular branching gives, for every $k\in\LieK$, a unique element $\iota(k)\in\LieK$ such that $\phi(\iota(k))=1\otimes k$. Thus $\iota\colon\LieK\hookrightarrow\LieK$ is an injective Lie homomorphism and
\[
 \phi^m(\iota^m(k))=1^{\otimes m}\otimes k.
\]
Moreover, $\deg\iota(k)=\lambda\deg k$ for homogeneous $k$. The proof of~\cite{bartholdi:wreathalg}*{Corollary~3.9} shows that the finite-codimension branching ideal $\LieK$ is finitely generated; in particular the graded vector space $\LieK/[\LieK,\LieK]$ is finite-dimensional. Choose finitely many homogeneous elements $S\subseteq\LieK$ whose images span it, so that $\LieK=\langle S\rangle_\Bbbk+[\LieK,\LieK]$. Since $\iota$ preserves $[\LieK,\LieK]$, and the degrees of $\iota^m(s)$ tend to infinity for $s\in S$, for all sufficiently large $m$ one has
\[
 \iota^m(\LieK)\subseteq[\LieK,\LieK].
\]
Consequently
\[
 1^{\otimes n}\otimes\iota^m(\LieK)
 \subseteq X^{\otimes n}\otimes[\LieK,\LieK]
 \subseteq\phi^n(\LieJ).
\]
Pulling back through the injective map $\phi^n$ gives a subalgebra of $\LieJ$ isomorphic to $\LieK$. Since $\LieJ\in\textsf{EL}_\beta$ and $\textsf{EL}_\beta$ is closed under subalgebras, we obtain $\LieK\in\textsf{EL}_\beta$.

It remains to recover $\LieL$ as a subquotient of $\LieK$. Recalling the subalgebra $\mathfrak D$ defined in~\eqref{eq:D}, define the section homomorphism
\[
 s=(\varepsilon\otimes\id)\phi\colon\mathfrak D\longrightarrow\LieL,
\]
where $\varepsilon\colon X\to\Bbbk$ is evaluation at $x=0$. Recurrence says that $s$ is surjective. For every ideal $\LieI\triangleleft\LieL$, put
\[
 \Sigma(\LieI)=s(\LieI\cap\mathfrak D).
\]
Then $\Sigma(\LieI)$ is an ideal of $\LieL$: if $y=s(b)$ with $b\in \LieI\cap\mathfrak D$ and $a\in\LieL$, choose $\widetilde a\in\mathfrak D$ with $s(\widetilde a)=a$ and note that $[a,y]=s([\widetilde a,b])$. Moreover, $\Sigma(\LieI)$ is a quotient of the subalgebra $\LieI\cap\mathfrak D$, so it is a subquotient of $\LieI$.

We claim that $\Sigma^3(\LieK)=\LieL$. Put $u_i=(\operatorname{ad}d)^i(v)$ for $1\leq i\leq p-1$. Directly from the recursion,
\[
 \phi(u_i)=\alpha_i x^{p-1-i}\otimes v
 \qquad(\alpha_i\in\Bbbk^\times),
\]
so $s(u_{p-1})$ is a nonzero multiple of $v$. On the other hand, if $u_1=[d,v]$, then
\[
 \phi(c)=\alpha_1x^{p-2}\otimes u_1.
\]
Thus $w=(\operatorname{ad}d)^{p-2}(c)$ lies in $\LieK\cap\mathfrak D$ and $s(w)$ is a nonzero multiple of $u_1$. Hence $u_1\in\Sigma(\LieK)$. Since $\Sigma(\LieK)$ is an ideal, it contains $u_{p-1}$, and therefore $v\in\Sigma^2(\LieK)$. The ideal $\Sigma^2(\LieK)$ again contains $u_{p-1}$, so $v\in\Sigma^3(\LieK)$; and since $v\in\mathfrak D$ and $s(v)=d$, we also have $d\in\Sigma^3(\LieK)$. As $\LieL=\langle d,v\rangle$, the claim follows.

Each application of $\Sigma$ takes a quotient of a subalgebra, and the subquotient relation is transitive. Hence $\LieL$ is a subquotient of $\LieK$. Since $\LieK\in\textsf{EL}_\beta$ and $\textsf{EL}_\beta$ is closed under subalgebras and quotients, this puts $\LieL$ in $\textsf{EL}_\beta$, contradicting the minimality of $\alpha$.
\end{proof}

%%%%%%%%%%%%%%%%%%%%%%%%%%%%%%%%%%%%%%%%%%%%%%%%%%%%%%%%%%%%%%%%
\section{An amenable Lie algebra of exponential growth}\label{ss:expgrowth}

We show that locally subexponential growth is not preserved under
extensions. This already fails when the kernel is locally finite (which we take to mean locally finite-dimensional, as usual) and the quotient is
one-dimensional:

\begin{thm}\label{thm:expogrowth}
There exists a finitely generated Lie algebra \(\LieL\) of exponential growth fitting in an exact sequence
\[
0\longrightarrow \LieK
 \longrightarrow \LieL
 \longrightarrow \Bbbk t
 \longrightarrow 0
\]
with \(\LieK\) locally finite.

In particular, $\LieL$ is amenable, being the extension of two elementarily amenable Lie algebras.
\end{thm}

Here is a sketch of the idea. If we can find an associative algebra $\AB$ of exponential growth and an ideal $\AJ\triangleleft\AB$ which is locally finite and satisfies $\AB/\AJ\cong\Bbbk[s]$, then we are done: there is a subalgebra $\LieL$ of $M_3(\AB)^{-}$ that maps onto $\Bbbk s\le\Bbbk[s]^{-}$ with kernel $\LieK$ contained in $M_3(\AJ)$, see~\S\ref{ss:liecomm}, and~\cite{alahmadi+:fg} for a similar construction.

Such an associative algebra $\AB$, produced as the monomial algebra associated with a well-chosen subshift, appears in~\cite{bell-greenfeld:monomial}*{\S6.2}. We give an alternative construction, based on a \emph{locally matrix} algebra~\cite{kuroshkin:locallymatrix}, obtained as the union of a tower of matrix algebras of doubly-exponentially-growing dimension, each of which is the endomorphism algebra of the previous one. This (associative) algebra is locally finite. We encode a sequence of elements that ``connect'' each matrix algebra to the next one, obtaining in this manner a two-generator algebra of endomorphisms.

\subsection{Growth in a locally matrix algebra}
We begin with a rank-one amplification observation.

\begin{lem}[Rank-one amplification]
\label{lem:rank-one-amplification}
Set \(B=M_d(\Bbbk)\), regard \(B\) as a subalgebra of
\(\End_\Bbbk(B)\) through the left regular representation
\[
a\mapsto \lambda_a,\qquad
\lambda_a(x)=ax,
\]
and let $\tr\colon B\rightarrow \Bbbk$ denote the ordinary matrix trace.  Define \(p\in\End_\Bbbk(B)\) by
\[
p(x)=\tr(x)1.
\]
Then
\[
\End_\Bbbk(B)
=
\langle\lambda_a p\lambda_b:a,b\in B\rangle_\Bbbk.
\]
Equivalently, after identifying \(B\) with its left regular image,
\[
\End_\Bbbk(B)=BpB.
\]
\end{lem}

\begin{proof}
For \(a,b,x\in B\),
\[
\lambda_a p\lambda_b(x)
=
a\,\tr(bx).
\]
The trace pairing
\[
B\times B\rightarrow \Bbbk,
\qquad
(b,x)\mapsto\tr(bx),
\]
is nondegenerate.  Hence the functionals
\[
x\mapsto\tr(bx),\qquad b\in B,
\]
span \(B^*\).  Allowing \(a\) to range over \(B\), the operators
\(\lambda_a p\lambda_b\) therefore span all rank-one endomorphisms of
the vector space \(B\).  Rank-one endomorphisms span
\(\End_\Bbbk(B)\).
\end{proof}

Define inductively
\[\A_0=M_2(\Bbbk),
\qquad
\A_{r+1}=\End_\Bbbk(\A_r),
\]
and embed \(\A_r\) into \(\A_{r+1}\) by the left
regular representation.  Thus we obtain a nested sequence
\[
\A_0\subseteq\A_1\subseteq\A_2\subseteq\cdots.
\]
Each \(\A_r\) is a full matrix algebra. Define
\[
p_r\in\A_{r+1}=\End_\Bbbk(\A_r)
\]
by
\[
p_r(x)=\tr(x)1.
\]
Lemma~\ref{lem:rank-one-amplification} gives
\begin{equation}
\label{eq:amplification}
\A_{r+1}
=
\A_r p_r\A_r.
\end{equation}

Set
\[
\A=\bigcup_{r\geq 0}\A_r,
\]
a locally finite associative algebra. Choose the standard matrix units \(E_{ij}\) in
\(\A_0=M_2(\Bbbk)\), and define a sequence
\((b_i)_{i\geq 0}\subseteq\A\) by
\[
b_0=1,\qquad
b_1=E_{12},\qquad
b_2=E_{21},\qquad
b_{r+3}=p_r\quad(r\geq 0).
\]
Assign to \(b_i\) the weight
\[
\operatorname{wt}(b_i)=i+1,
\]
and \(n\geq 0\) set
\[P(n)=\Big\langle b_{i_1}\cdots b_{i_m}:\sum_{\ell=1}^m(i_\ell+1)\leq n\Big\rangle_\Bbbk\le\A.\] 
The empty product is allowed.

\begin{lem}[Exponential weighted profile growth]

For all sufficiently large \(n\),
\[
\dim_\Bbbk P(n)
\geq
\exp\left(\frac{\log 4}{20}\,n\right).
\]
In particular, the weighted algebra generated by the sequence
\((b_i)\) has exponential growth.
\end{lem}

\begin{proof}
Set
\[
s_0=5,
\qquad
s_{r+1}=2s_r+r+4.
\]
We claim that
\begin{equation}
\label{eq:stage-contained-in-profile-ball}
\A_r\subseteq P(s_r)
\end{equation}
for every \(r\geq 0\).

For \(r=0\), the algebra \(\A_0=M_2(\Bbbk)\) is spanned by
\[
1=b_0,\qquad
E_{12}=b_1,\qquad
E_{21}=b_2,
\]
and
\[
E_{11}=E_{12}E_{21}=b_1b_2,
\qquad
E_{22}=E_{21}E_{12}=b_2b_1.
\]
All these elements have weight at most \(5\), so
\(\A_0\subseteq P(5)\).

Suppose now that \(\A_r\subseteq P(s_r)\).  By
\eqref{eq:amplification}, we have $\A_{r+1}=\A_r p_r\A_r$, and since \(p_r=b_{r+3}\) has weight \(r+4\), we get
\[
\A_{r+1}
\le
P(2s_r+r+4)
=
P(s_{r+1}).
\]
This proves \eqref{eq:stage-contained-in-profile-ball}.

The recurrence for \(s_r\) has the explicit solution
\begin{equation}
\label{eq:profile-radius}
s_r=10\cdot 2^r-r-5.
\end{equation}
On the other hand,
\[
\dim_\Bbbk\A_0=4,
\qquad
\dim_\Bbbk\A_{r+1}
=
\bigl(\dim_\Bbbk\A_r\bigr)^2,
\]
and therefore
\[
\dim_\Bbbk\A_r=4^{2^r}.
\]

Let \(n\geq 5\), and choose \(r\) such that
\[
s_r\leq n<s_{r+1}.
\]
Then
\[
\dim_\Bbbk P(n)
\geq
\dim_\Bbbk\A_r
=
4^{2^r}.
\]
Moreover, by \eqref{eq:profile-radius},
\[
n<s_{r+1}
=
20\cdot 2^r-r-6
<
20\cdot 2^r.
\]
Hence \(2^r>n/20\), and therefore
\[
\dim_\Bbbk P(n)
\geq
4^{2^r}
>
4^{n/20}
=
\exp\left(\frac{\log 4}{20}\,n\right).\qedhere
\]
\end{proof}

\subsection{The shift-profile algebra}

Let
\[
V=\bigoplus_{j\geq 0} e_j\A
\]
be the free right \(\A\)-module with basis
\(e_0,e_1,\ldots\).  Define two right \(\A\)-linear
endomorphisms \(S,C\in\End_{\A}(V)\) by
\[
S(e_j)=e_{j+1},
\qquad
C(e_j)=e_0b_j.
\]
Thus \(S\) is the unilateral shift, while \(C\) stores the entire
profile \(b_0,b_1,\ldots\) in its zeroth row.  Since \(b_0=1\),
\[
C^2=C.
\]

Let
\[
\AB=\Bbbk\langle S,C\rangle
\subseteq\End_{\A}(V)
\]
be the unital associative algebra generated by \(S\) and \(C\). For \(p,q\geq 0\) and \(u\in\A\), define \(\rho_{p,u,q}\in\End_{\A}(V)\) by
\[
\rho_{p,u,q}(e_j)=e_pu b_{q+j}.
\]

\begin{lem}[Normal form for profile words]
\label{lem:profile-normal-form}
Every monomial in \(S,C\) containing at least one occurrence of \(C\)
is of the form \(\rho_{p,u,q}\), for suitable \(p,q\geq 0\) and
\(u\in\A\).

Moreover,
\begin{equation}
\label{eq:rho-product}
\rho_{p,u,q}\rho_{p',u',q'}
=
\rho_{p,\;u b_{q+p'}u',\;q'}.
\end{equation}
\end{lem}

\begin{proof}
Every monomial containing \(C\) can be written as
\[
S^p C S^{i_1}C\cdots S^{i_m}C S^q
\]
for some \(p,q,i_1,\ldots,i_m\geq 0\), where \(m\) may be zero.
Applying this operator to \(e_j\) gives
\(
e_p b_{i_1}\cdots b_{i_m}b_{q+j}
\), so it is equal to
\(
\rho_{p,\;b_{i_1}\cdots b_{i_m},\;q}\).

For the multiplication formula, right \(\A\)-linearity gives
\[
\begin{aligned}
\rho_{p,u,q}\rho_{p',u',q'}(e_j)
&=
\rho_{p,u,q}
\bigl(e_{p'}u'b_{q'+j}\bigr)\\
&=
e_pu b_{q+p'}u'b_{q'+j},
\end{aligned}
\]
which is precisely
\(\rho_{p,\;u b_{q+p'}u',\;q'}(e_j)\).
\end{proof}

Let \(\AJ\triangleleft\AB\) be the two-sided ideal
generated by \(C\).

\begin{lem}[The locally finite profile ideal]
\label{lem:locally-finite-profile-ideal}
The associative algebra \(\AJ\) is locally finite, and
\[
\AB/\AJ\simeq \Bbbk[s],
\]
under the quotient map
\[
\pi\colon\AB\rightarrow \Bbbk[s],
\qquad
\pi(S)=s,\qquad
\pi(C)=0.
\]
\end{lem}

\begin{proof}
Every word in \(S,C\) either is a power of \(S\), or contains \(C\)
and hence belongs to \(\AJ\).  Thus
\[
\AB=\Bbbk[S]+\AJ.
\]

By Lemma~\ref{lem:profile-normal-form}, every element of
\(\AJ\) is a finite linear combination of operators
\(\rho_{p,u,q}\).  Each such operator has only one nonzero row,
namely row \(p\).  Consequently every element of \(\AJ\) has
only finitely many nonzero rows.

By contrast, if
\[
f(S)=\alpha_0+\alpha_1S+\cdots+\alpha_dS^d
\]
is a nonzero polynomial in \(S\), then
\[
f(S)e_j=\sum_{i=0}^d\alpha_i e_{j+i}
\]
for every \(j\geq 0\).  Thus \(f(S)\) has infinitely many nonzero
rows.  It follows that
\[
\Bbbk[S]\cap\AJ=0.
\]
Hence
\[
\AB=\Bbbk[S]\oplus\AJ
\]
as vector spaces, and the quotient is naturally \(\Bbbk[s]\).

It remains to prove local finiteness.  Let
\(F\subseteq\AJ\) be finite.  Express the elements of \(F\) as
finite linear combinations of operators \(\rho_{p,u,q}\).  There are
finite sets
\[
P,Q\subseteq\mathbb N
\]
and a finite set \(U\subseteq\A\) such that every term
appearing in these expressions has
$p\in P$, $q\in Q$, $u\in U$. Let \(\mathscr U\subseteq\A\) be the associative subalgebra
generated by
$U\cup
\{b_{q+p}:q\in Q,\ p\in P\}$.
Since \(\A\) is locally finite, \(\mathscr U\) is
finite-dimensional.

Consider the finite-dimensional subspace
\[
\mathscr W=\left\langle\rho_{p,u,q}:p\in P,\ q\in Q,\ u\in\mathscr U\right\rangle_\Bbbk.
\]
Formula \eqref{eq:rho-product} shows that \(\mathscr W\) is closed under
multiplication: indeed,
\[
\rho_{p,u,q}\rho_{p',u',q'}
=
\rho_{p,\;u b_{q+p'}u',\;q'},
\]
and \(u b_{q+p'}u'\in\mathscr U\).  Thus the associative subalgebra
generated by \(F\) is contained in the finite-dimensional algebra
\(\mathscr W\).  Therefore \(\AJ\) is locally finite.
\end{proof}

The shift reads the profile with exactly the intended weighted cost.

\begin{lem}[Reading the profile]
\label{lem:reading-the-profile}
For \(i_1,\ldots,i_m\geq 0\), put
\[
U(i_1,\ldots,i_m)
=
C S^{i_1}C S^{i_2}\cdots C S^{i_m}C.
\]
Then
\begin{equation}
\label{eq:profile-read}
U(i_1,\ldots,i_m)(e_0)
=
e_0b_{i_1}\cdots b_{i_m}.
\end{equation}
Moreover, the word length of \(U(i_1,\ldots,i_m)\) in the generators
\(S,C\) is
\begin{equation}
\label{eq:associative-profile-cost}
1+\sum_{\ell=1}^m(i_\ell+1).
\end{equation}
\end{lem}

\begin{proof}
Equation \eqref{eq:profile-read} follows immediately by reading the
operator product from right to left.  The word contains \(m+1\)
copies of \(C\) and \(\sum_{\ell=1}^m i_\ell\) copies of \(S\), which
gives \eqref{eq:associative-profile-cost}.
\end{proof}

\subsection{Encoding associative products by Lie commutators}\label{ss:liecomm}

For an associative algebra \(B\), write \(B^{(-)}\) for the Lie
algebra with bracket
\[
[a,b]=ab-ba.
\]
For \(r\in\AB\), let \(E_{ij}(r)\) denote the elementary
\(3\times3\) matrix with entry \(r\) in position \((i,j)\).

In \(M_3(\AB)^{(-)}\), define
\[
t=E_{22}(S),\qquad
x=E_{12}(C),\qquad
y=E_{23}(C),\qquad
z=E_{32}(C).
\]

\begin{lem}[Matrix-unit encoding]
\label{lem:matrix-unit-encoding}
For \(r\in\AB\),
\begin{align}
[E_{12}(r),t]
&=E_{12}(rS),
\label{eq:append-S}\\
[[E_{12}(r),y],z]
&=E_{12}(rC).
\label{eq:append-C}
\end{align}

For a tuple
\[
\mathbf i=(i_1,\ldots,i_m)\in\mathbb N^m,
\]
there is consequently a left-normed Lie word \(X_{\mathbf i}\) in
\(t,x,y,z\) such that
\begin{equation}
\label{eq:encoded-profile-word}
X_{\mathbf i}
=
E_{12}
\left(
C S^{i_1}C S^{i_2}\cdots C S^{i_m}C
\right),
\end{equation}
and whose Lie length is
\begin{equation}
\label{eq:encoded-profile-length}
1+\sum_{\ell=1}^m(i_\ell+2).
\end{equation}
\end{lem}

\begin{proof}
The elementary-matrix multiplication rule gives
\[
E_{ij}(a)E_{kl}(b)
=
\delta_{jk}E_{il}(ab).
\]
Therefore
\[
[E_{12}(r),E_{22}(S)]
=
E_{12}(rS),
\]
which proves \eqref{eq:append-S}.  Similarly,
\[
[E_{12}(r),E_{23}(C)]
=
E_{13}(rC),
\]
and hence
\[
[E_{13}(rC),E_{32}(C)]
=
E_{12}(rC^2)
=
E_{12}(rC),
\]
using \(C^2=C\).  This proves \eqref{eq:append-C}.

Starting from \(x=E_{12}(C)\), apply the operation
\[
u\mapsto [u,t]
\]
exactly \(i_1\) times, then apply
\[
u\mapsto [[u,y],z],
\]
then repeat with \(i_2,\ldots,i_m\).  Equations
\eqref{eq:append-S} and \eqref{eq:append-C} give
\eqref{eq:encoded-profile-word}.  The initial \(x\) contributes one
letter; the \(\ell\)-th stage contributes \(i_\ell\) copies of \(t\)
and one copy each of \(y,z\).  This gives
\eqref{eq:encoded-profile-length}.
\end{proof}

\begin{proof}[Proof of Theorem~\ref{thm:expogrowth}]
Define
\[
\LieL
=
\operatorname{Lie}\langle t,x,y,z\rangle
\le M_3(\AB)^{(-)}.
\]
The quotient homomorphism
\[
\pi\colon\AB\rightarrow \Bbbk[s]
\]
from Lemma~\ref{lem:locally-finite-profile-ideal} induces a Lie
homomorphism
\[
M_3(\pi)\colon
M_3(\AB)^{(-)}
\rightarrow
M_3(\Bbbk[s])^{(-)}.
\]
On the generators of \(\LieL\), it satisfies
\[
M_3(\pi)(t)=E_{22}(s),
\qquad
M_3(\pi)(x)=M_3(\pi)(y)=M_3(\pi)(z)=0.
\]
Thus the image of \(\LieL\) is the one-dimensional abelian Lie
algebra
\[
\Bbbk E_{22}(s).
\]

Define
\[
\LieK
=
\ker\left(M_3(\pi)|_{\LieL}\right).
\]
Since \(\ker\pi=\AJ\), we have $\LieK
=
\LieL\cap M_3(\AJ)$, so we have an exact sequence
\[
0\longrightarrow\LieK
\longrightarrow\LieL
\longrightarrow \Bbbk E_{22}(s)
\longrightarrow0.
\]
The sequence splits through the one-dimensional subalgebra \(\Bbbk t\);
in particular,
\[
\LieL=\LieK\rtimes \Bbbk t.
\]

We claim that \(\LieK\) is locally finite.  More generally,
\(M_3(\AJ)^{(-)}\) is locally finite.  Indeed, given finitely
many matrices in \(M_3(\AJ)\), collect their finitely many
entries.  Since \(\AJ\) is locally finite as an associative
algebra, those entries generate a finite-dimensional associative
subalgebra \(\mathscr U\subseteq\AJ\).  The given matrices
then lie in the finite-dimensional Lie algebra $M_3(\mathscr U)^{(-)}$.

It remains to prove that \(\LieL\) has exponential growth.  Put
\[
F=\Bbbk\{t,x,y,z\}.
\]
For \(n\geq0\), let \(W_n\) be the span of the elements
\(X_{\mathbf i}\) from Lemma~\ref{lem:matrix-unit-encoding}, where
$\mathbf i=(i_1,\ldots,i_m)$
ranges over all finite tuples satisfying
\[
\sum_{\ell=1}^m(i_\ell+1)\leq n.
\]
By \eqref{eq:encoded-profile-length},
\[
\text{Lie length}(X_{\mathbf i})=1+\sum_{\ell=1}^m(i_\ell+2)\leq 2n+1,
\]
because \(m\leq n\).  Consequently, $W_n$ lies in the ball $B_{2n+1}$ of radius $2n+1$ of $\LieL$.

Let
\[
\operatorname{pr}_0\colon
V\rightarrow e_0\A
\]
be the projection onto the zeroth summand.  For \(r\in\AB\),
write
\[
\operatorname{pr}_0(r(e_0))=e_0\theta(r),
\qquad
\theta(r)\in\A.
\]
This defines a linear map
$\theta\colon\AB\rightarrow\A$.
By Lemma~\ref{lem:reading-the-profile},
\[
\theta
\left(
C S^{i_1}C\cdots S^{i_m}C
\right)
=
b_{i_1}\cdots b_{i_m},
\]
so the linear map $E_{12}(r)\mapsto\theta(r)$
sends \(W_n\) onto a subspace containing \(P(n)\).  Hence
\[
\dim_\Bbbk W_n\geq\dim_\Bbbk P(n).
\]
We obtain
\[\dim_\Bbbk B_{2n+1}\ge\dim_\Bbbk P(n)\ge\exp\left(\frac{\log4}{20}\,n\right),
\]
so \(\LieL\) has exponential growth.
\end{proof}

%%%%%%%%%%%%%%%%%%%%%%%%%%%%%%%%%%%%%%%%%%%%%%%%%%%%%%%%%%%%%%%%
\section{Associated graded Hopf-module coalgebras}\label{ss:graded}

Let $H$ be a cocommutative Hopf algebra with augmentation ideal $\varpi=\ker\varepsilon$. Put $H_n=\varpi^n$. These spaces form a Hopf filtration: $H_iH_j\subseteq H_{i+j}$, $\Delta(H_n)\subseteq\sum_{i+j=n}H_i\otimes H_j$, and $S(H_n)\subseteq H_n$. Thus $\gr H=\bigoplus_{n\ge0}H_n/H_{n+1}$ is a graded Hopf algebra.

If $M$ is a left $H$-module coalgebra, put $M_n=H_nM$. Then $\Delta(M_n)\subseteq\sum_{i+j=n}M_i\otimes M_j$, so $\gr M=\bigoplus_{n\ge0}M_n/M_{n+1}$ is naturally a graded $\gr H$-module coalgebra. The filtration is \emph{separated} when $\bigcap_nM_n=0$.

\begin{lem}\label{lem:tensor-separated-quotient}
Let $V=V_0\supseteq V_1\supseteq\cdots$ be a descending filtration and put $V_\infty=\bigcap_nV_n$. Then
\[
 \bigcap_{n\ge0}\sum_{i+j=n}V_i\otimes V_j
 =V_\infty\otimes V+V\otimes V_\infty.
\]
\end{lem}

\begin{proof}
The inclusion from right to left is immediate. For the converse, quotient by $V_\infty$, so it is enough to assume $\bigcap_nV_n=0$. If $0\ne z\in V\otimes V$, choose finite-dimensional $P,Q\le V$ with $z\in P\otimes Q$. The descending chains $P\cap V_n$ and $Q\cap V_n$ eventually vanish, say $P\cap V_r=Q\cap V_s=0$. Hence $P\otimes Q\to(V/V_r)\otimes(V/V_s)$ is injective. If $n\ge r+s-1$, every summand $V_i\otimes V_j$ with $i+j=n$ maps to zero in this quotient, so $z$ cannot belong to $\sum_{i+j=n}V_i\otimes V_j$.
\end{proof}

\begin{lem}\label{lem:graded-separated}
Assume that the augmentation filtration on an amenable $H$-module coalgebra $M$ is separated. Then $\gr M$ is amenable over $\gr H$.
\end{lem}

\begin{proof}
Let $\overline F\le\gr H$ be finite-dimensional and $\epsilon>0$. Enlarging $\overline F$ if necessary, assume it is homogeneous, and choose degree by degree a finite-dimensional space $F\le H$ whose leading symbols span $\overline F$. By Theorem~\ref{thm:hopf}, choose a finite-dimensional $E\le M$ such that, for $E^+=E+FE$, one has $\dim E^+<(1+\epsilon)\dim E$.

Give $E$ and $E^+$ the filtrations induced from $M$. Since the filtration is separated and these spaces are finite-dimensional, $\dim\gr E=\dim E$ and $\dim\gr E^+=\dim E^+$. Filtered multiplication gives $\overline F\gr E\subseteq\gr E^+$, and therefore
\[
 \dim(\gr E+\overline F\gr E)
 \le\dim E^+<(1+\epsilon)\dim\gr E.
\]
Thus $\gr M$ is algebraically amenable; Theorem~\ref{thm:hopf} rounds the Følner subspace to a finite-dimensional subcoalgebra.
\end{proof}

\begin{thm}\label{thm:graded}
If $M$ is an amenable left $H$-module coalgebra, then $\gr M$ is an amenable $\gr H$-module coalgebra.
\end{thm}

\begin{proof}
Put $M_\infty=\bigcap_nM_n$. Since each $M_n$ is an $H$-submodule, so is $M_\infty$. Moreover $M_\infty\subseteq M_1=\varpi M$, hence $\varepsilon_M(M_\infty)=0$. For $x\in M_\infty$, the coalgebra filtration gives
\[
 \Delta(x)\in\bigcap_n\sum_{i+j=n}M_i\otimes M_j
 =M_\infty\otimes M+M\otimes M_\infty
\]
by Lemma~\ref{lem:tensor-separated-quotient}. Thus $M_\infty$ is an $H$-stable coideal, and $M^{\mathrm{sep}}=M/M_\infty$ is an $H$-module coalgebra. Theorem~\ref{thm:module-coalgebra-quotient} shows that it is amenable.

Its augmentation filtration is $M_n^{\mathrm{sep}}=M_n/M_\infty$, and is separated. Moreover $(M_n/M_\infty)/(M_{n+1}/M_\infty)\cong M_n/M_{n+1}$ compatibly with the graded action and comultiplication, so $\gr M^{\mathrm{sep}}\cong\gr M$. Lemma~\ref{lem:graded-separated} now proves the theorem.
\end{proof}

\end{document}